\documentclass[11pt]{article}

\usepackage{color}
\usepackage{a4wide}
\usepackage{amsmath}
\usepackage{bbm}
\usepackage{eucal} 
\usepackage{amssymb}
\usepackage{amsthm}
\usepackage[english, algosection, algoruled, noline]{algorithm2e}
\usepackage{enumerate}
\usepackage{kbordermatrix} 

\newcommand{\rr}{\mathcal{R}} %The range space
\newcommand{\kk}{\mathcal{N}} %The null space
\newcommand{\ene}{\mathbbmss{N}} %The set of natural numbers
\newcommand{\ce}{\mathbbmss{C}} %The set of complex numbers
\newcommand{\erre}{\mathbbmss{R}} %The set of real numbers
\newcommand{\matriz}[2]{\ce_{#1,#2}} %The set of mxn complex matrices
\newcommand{\mat}[4]
{\left[ \begin{array}{cc} #1 & #2 \\ #3 & #4 \end{array} \right] } %A 2x2 matrix
\renewcommand{\S}{\mathcal{S}} %A semigroup
\DeclareMathOperator{\rk}{rk} %The rank
\DeclareMathOperator{\tr}{tr} %The trace
\newcommand{\dn}{\mathbf{d}} %The vector u
\newcommand{\en}{\mathbf{e}} %The vector u
\newcommand{\on}{\mathbf{0}} %The null vector 
\newcommand{\un}{\mathbf{u}} %The vector u
\newcommand{\vn}{\mathbf{v}} %The vector v
\newcommand{\wn}{\mathbf{w}} %The vector w
\newcommand{\xn}{\mathbf{x}} %The vector x
\newcommand{\yn}{\mathbf{y}} %The vector y
\newcommand{\mary}[2]{#1^{\| #2}} %A generalization of the inverse introduced by Mary 
\newcommand{\core}[1]{{#1}^{\tiny{\textcircled{\#}}}} %The core inverse
\newcommand{\cored}[1]{{#1}_{\tiny{\textcircled{\#}}}} %The dual core inverse
\newcommand{\N}{\mathcal{N}}
\newcommand{\M}{\mathcal{M}}
\newcommand{\T}{\mathcal{T}}
\newcommand{\ADE}{A^{\parallel (D,E)}}

\newtheorem{df}{Definition}[section]
\newtheorem{thm}[df]{Theorem}
\newtheorem{cor}[df]{Corollary}
\newtheorem{rema}[df] {Remark}
\newtheorem{lem}[df] {Lemma}
\newtheorem{pro}[df] {Proposition}

\title{On one-sided $(B,C)$-inverses of arbitrary matrices}

\author{Julio Ben\'{\i}tez, Enrico Boasso, Hongwei Jin}

\date{}

\begin{document}
\maketitle

%\tableofcontents

\begin{abstract}\noindent In this article one-sided $(b, c)$-inverses of  arbitrary matrices as well as one-sided inverses along a (not necessarily square) matrix, 
will be studied. In adddition, the $(b, c)$-inverse and the inverse along an element will be also researched in the context of 
rectangular matrices. 
\end{abstract}

%%%%%%%%%%%%%%%%%%%%%%%%%%%%%%%%%%%%

\section{Introduction and notation}

Several generalized inverses have been studied in the literature. Recently
two important outer inverses have been introduced: the inverse along an element (see \cite{Mary})
and the $(b,c)$-inverse (see \cite{Drazin}). In fact, these two generalized  inverses encompass some of the most important
outer inverses such as the group inverse, the Drazin inverse and the Moore-Penrose inverse.  
Furthermore, in the context of semigroups the left and right inverses along an element were defined in \cite{zhu};
these notions extend the inverse along an element.  Similarly, in the frame of rings, left and right $(b, c)$-invertible elements were introduced in \cite{KVC}; 
these definitions  extend both the $(b, c)$-inverse and the left and right inverses along an element.\par

\indent As it has been said, the aforementioned outer inverses and their extensions were defined in semigroups or rings.
However, observe that the set of $n \times m$ complex matrices is not a semigroup (unless $n=m$).
The main purpose of this article is to extend the above mentioned (one-sided) inverses as well as the $(b, c)$-inverse and the inverse along an element to arbitrary 
matrices and to study their basic properties. Naturally, the results presented also hold for square matrices.

\indent In section 3, after having recalled the main notions considered in this article in section 2, the one-sided $(b, c)$-inverses
and the left and right inverses along an element in the context of arbitrary matrices will be thoroughly studied. In sections 4 and 5 the $(b, c)$-inverse and the inverse along
an element will be introduced and studied in the same frame, respectively. In section 6 it will be characterized when the generalized inverses introduced in sections 4 and 5 are inner inverses.
In section 7 the relationships among the notions considered in sections 4 and 5 and the outer inverse
with prescribed range and null space will be studied. In section 8 the continuity and the differentiability of the notions introduced in sections 4 and 5 will be considered.
Finally, in section 9 algorithms to compute the $(b,c)$-inverse in the matrix frame will be given. \par 

\indent Before going on, the definition of several generalized inverses in the context of rings will be given. The corresponding definitions for
complex matrices can be obtained making obvious changes.
Let $\rr$ be a unitary ring and $a \in \rr$. 
\begin{enumerate}[(i)]
\item The element $a$ is said to be {\em group invertible}, if there exists $x \in \rr$ such that $axa=a$, $xax=x$, 
	and $ax=xa$. This $x$ is unique and it is denoted by $a^\#$.
\item The element $a$ is said to be {\em Drazin invertible}, if there exists $x \in \rr$ such that $xax=x$, $xa=ax$, 
	and $a^{n+1}x=a^n$, for some $n \in \ene$. This $x$ is unique and it is denoted by $a^d$. Note that
when $n=1$, the group inverse is obtained (see \cite{Drazin58}).
\item Let $\rr$ have an involution. The element $a$ is said to be {\em Moore-Penrose invertible}, if exists 
	$x \in \rr$ such that $axa=a$, $xax=x$, $(ax)^*=ax$, and $(xa)^*=xa$. 
	This $x$ is unique and it is denoted by $a^\dag$ (see \cite{penrose}).
\item Let $\mathcal{R}$ have an involution and let $m, n \in \rr$ be invertible 
	Hermitian elements in $\rr$. The element $a \in \mathcal{R}$ is said to be {\em Moore-Penrose invertible 
	with weights $m, n$}, if there exists $x \in \rr$ such that $axa = a$, $xax = x$, 
	$(max)^* = max$, $(nxa)^* = nxa$. This $x$ is unique and it is denoted by $a^\dag_{m,n}$. 
	In a ring $\rr$ with an involution, an element $u \in \rr$ is said to be {\em positive}, if 
	there exists a Hermitian $v \in \rr$ such that $u=v^2$.
\item Let $\rr$ have an involution. The element $a$ is said to be {\em core invertible}, 
	if there exists $x \in \rr$ such that $axa=a$, $x\rr = a\rr$ and $\rr x=\rr a^*$.
	This $x$ is unique and it is denoted by $\core{a}$ (see \cite{bt, rdd}).
\item Let $\rr$ have an involution. The element $a$ is said to be {\em dual core invertible}, 
	if there exists $x \in  \rr$ such that $axa=a$, $x \rr = a^* \rr$ and $\rr x = \rr a$.
	This $x$ is unique and it is denoted by $\cored{a}$ (see \cite{bt, rdd}).
\end{enumerate} 

\indent To end this section, some notation is introduced. Let $m$, $n\in\ene$ and 
denote by $\matriz{m}{n}$ the set of $m \times n$ complex matrices. The symbol $\ce_n$
will stand for $\matriz{n}{n}$.
Any vector of the space $\ce^n$ will be considered as a column vector, i.e.,  
$\ce^n$ will be identified with $\matriz{n}{1}$. 

Moreover, $I_n$ will mean the identity matrix of order $n$, 
$\rk(A)$ the rank  of $A \in \matriz{m}{n}$, and when $n=m$, $\tr(A)$ will stand for  the trace of $A$. Related to a matrix $A \in 
\matriz{m}{n}$ there are two linear subspaces, the column space and the null space, which are defined 
respectively by
$$
\rr(A) = \{ A\xn : \xn \in \ce^n\}, \qquad \kk(A) = \{ \xn \in \ce^n: A \xn = \on \}.
$$

\noindent Recall that $\rk(A) + \dim \kk(A) = n$. 
Given a linear mapping $f:\ce^n \to \ce^m$,
the subsets $\rr(f)$ and $\kk(f)$ are defined in a similar way. Observe that if $A$ is the
matrix associated to $f$ respect with the standard basis, then $\rr(A)=\rr(f)$ and $\kk(A)
= \kk(f)$.

In addition, the conjugate transpose of the matrix $A$ will be denoted by $A^*$. Two basic equalities are
$\kk(A^*)=\rr(A)^\perp$ and $\rr(A^*)=\kk(A)^\perp$, for $A \in \matriz{n}{m}$. 

If $\M$ is a subspace of $\ce^n$, the symbol $I_\M$ will stand for the identity linear transformation on $\M$
and $P_\M$ for the orthogonal
projector onto $\M$. When $\N$ and $\M$ are two subspaces of $\ce^n$,
$P_{\M, \N}$ will stand for the idempotent whose range is $\M$ and whose null space is $\N$.

\indent Recall that given $X\in \matriz{n}{m}$, $Y\in \matriz{m}{n}$ is an \it inner inverse of $A$, \rm if $XYX=X$.
In addition, $Y$ is said to be an \it outer inverse of $A$, \rm if $YXY=Y$. Next  the outer inverse with prescribed range and null space
will be recalled.

Let $A\in\matriz{n}{m}$ and consider subspaces $\T\subseteq \ce^m$ and 
$\S\subseteq\ce^n$ such that $\dim \T=s\le \rk(A)$ and $\dim \S=n-s$. Necessary and sufficient for the matrix $A$ to have an outer inverse $Z$ such that $\rr(Z)=\T$ and $\kk(Z)=\S$
is that $A(\T)\oplus \S=\ce^n$, in which case $Z$ is unique and it is 
denoted by $A^{(2)}_{\T, \S}$ (see for example \cite[Lemma 1.1]{W}).  

%%%%%%%%%%%%%%%%%%%%%%%%%%%%%%%%%%%%%%%%%%%%

\section{The definition of the one-sided $(D,E)$ inverses and their relationship with other inverses}

In first place the definition of the $(b, c)$-inverse will be recalled (see \cite[Definition 1.3]{Drazin}).

\begin{df}\label{def1}
Let $\S$ be a semigroup and consider $a$, $b$, $c\in\S$. The element $y\in\S$
will be said to be {\em the $(b,c)$-inverse} of $a$, if the following 
equations hold:
\begin{enumerate}[{\rm (i)}]
\item $y\in(b\S y)\cap (y\S c)$.
\item $b=yab$, $c=cay$.
\end{enumerate}
\end{df}

According to \cite[Theorem 2.1]{Drazin}, if the element $y$ in Definition \ref{def1} exists, then
it is unique. In this case, the element under consideration will be denoted by $a^{\parallel (b, c)}$. As it was pointed out in \cite{Drazin}, this inverse generalizes among others the standard inverse, the Drazin inverse, 
and the Moore-Penrose inverse. To learn more on this inverse, see \cite{B, CCW, Drazin, Drazin2, KC}.

\indent The inverse along an element was introduced in  \cite[Definition 4]{Mary}. Next its definition will be recalled.\par

\begin{df}\label{def1000}
Let $\S$ be a semigroup. An element $a \in \S$ is said to be {\em invertible along} $d \in \S$ 
if there exists $y \in \S$ such that
\begin{enumerate}[{\rm (i)}]
\item $yad = d = day$.
\item $y\S \subseteq d\S$.
\item $\S y \subseteq \S d$.
\end{enumerate}
\end{df}

\indent According to \cite[Theorem 6]{Mary}, if the element $y \in \S$ in Definition \ref{def1000} exists, then it is unique. This element  is denoted by $\mary{a}{d}$. It is worth noting that according to \cite[Proposition 6.1]{Drazin}, the inverse along an element is a particular
case of the $(b, c)$-inverse, i.e., the $(d, d)$-inverse coincides with the inverse along $d$. The outer inverses recalled in Definition \ref{def1} and Definition \ref{def1000} encompass several 
generalized inverses, as the following two theorems show.

\begin{thm}{\rm (\cite[Theorem 11]{Mary})}
Let $\mathcal{S}$ be a semigroup and let $a \in \mathcal{S}$.
\begin{enumerate}[{\rm (i)}]
\item If $\mathcal{S}$ has a unity, then $a$ is invertible if and only if
	$a$ is invertible along $1$. In this case $a^{-1} = \mary{a}{1}$.
\item $a$ is group invertible if and only if $a$ is invertible along $a$.
	In this case $a^\# = \mary{a}{a}$.
\item $a$ is Drazin invertible if and only if $a$ is invertible along
	$a^m$ for some $m \in \mathbbmss{N}$. In this case $a^D = \mary{a}{a^m}$.
\item If $\mathcal{S}$ is a $*$-semigroup, $a$ is Moore-Penrose invertible if and only if 
	$a$ is invertible along $a^*$. In this case $a^\dag = \mary{a}{a^*}$.
\end{enumerate}
\end{thm}

\begin{thm}\label{thmA}
Let $\mathcal{R}$ be a ring with an involution and $a \in \mathcal{R}$.
\begin{enumerate}[{\rm (i)}]
\item {\rm (\cite[Theorem 4.3]{rdd})}
	If $a$ is Moore-Penrose invertible, then $a$ is core invertible if and only
	if it is invertible along $aa^*$. In this case the inverse along $aa^*$ coincides
	$\core{a}$.
\item {\rm (\cite[Theorem 4.3]{rdd})}
	If $a$ is Moore-Penrose invertible, then $a$ is dual core invertible if and only if it is
	invertible along $a^*a$. In this case the inverse along $a^*a$ coincides with $\cored{a}$.
\item {\rm (\cite[Theorem  3.2]{bb})}
	If $m$, $n\in \mathcal{R}$ are invertible and positive, then $a$ is weighted
	Moore-Penrose invertible with weights $m$ and $n$ if and only if
	$a$ is invertible along $n^{-1} a^* m$.
	In this case, the inverse along $n^{-1}a^*m$ coincies with $a^\dagger_{m,n}$.
\end{enumerate}
\end{thm}

\indent Recently the inverse along an element and the $(b, c)$-inverse were extended 
by means of one-sided inverses. Next follow the corresponding definitions. See \cite[Definition 2.1]{KVC} and 
\cite[Definition 2.1]{zhu}.

\begin{df} \label{def1100}
Let $\rr$ be a ring and 
let $b, c \in \rr$. An element $a \in \rr$ is said to be {\em left $(b,c)$-invertible}, 
if there exists $y \in \rr$ such that 
\begin{enumerate}[{\rm (i)}]
\item $yab = b$.
\item $\rr y \subseteq \rr c$.
\end{enumerate}
In this case $y$ is called a {\em left $(b,c)$-inverse} of $a$.

\noindent An element $a \in \rr$ is {\em right $(b,c)$-invertible}, if there exists 
$y \in \rr$ such that
\begin{enumerate}[{\rm (i)}] \setcounter{enumi}{2}
\item $cay = c$.
\item $y \rr \subseteq b\rr$.
\end{enumerate}
In this case $y$ is called a {\em right $(b, c)$-inverse} of $a$.
\end{df}

\indent Recall that given $a,b,c$ elements in a ring $\rr$, according to \cite[Corollary 3.7]{KVC},
$a$ is $(b, c)$-invertible if and only if it is both left and right $(b, c)$-invertible. When in Definition \ref{def1100},
$b=c$, the one-sided inverses along an element are obtained. 

\begin{df}\label{d24}
Let $\mathcal{S}$ be a semigroup. An element $a \in \mathcal{S}$ is {\em left invertible along} 
$d \in \mathcal{S}$, if there exists $y \in \mathcal{S}$ such that
\begin{enumerate}[{\rm (i)}]
\item $yad = d$.
\item $\mathcal{S} y \subseteq \mathcal{S} d$.
\end{enumerate}
\noindent In this case $y$ is called a {\rm left inverse of $a$ along $d$}.\par
\noindent An element $a \in \mathcal{S}$ is {\em right invertible along} 
$d \in \mathcal{S}$, if there exists $y \in \mathcal{S}$ such that
\begin{enumerate}[{\rm (i)}] \setcounter{enumi}{2}
\item $day = d$.
\item $y \mathcal{S} \subseteq d \mathcal{S}$.
\end{enumerate}
\noindent In this case $y$ is called a {\rm right inverse of $a$ along $d$}.
\end{df}

\indent Recall that given a semigroup $\mathcal{S}$ and  $a$, $d\in\mathcal{S}$, according to \cite[Corollary 2.5]{zhu},
$a$ is invertible along $d$ if and only if $a$ is left and right invertible along $d$.\par

\indent Naturally, since all the inverses that have been considered in this section up to now have been defined in semigroups
and rings, they can not be applied to matrices, unless they are square. However, to extend the aforementioned notions to arbitrary
 matrices, first it is necessary to recall the following facts. Let $U$, $V\in  \matriz{m}{n}$. There is $X\in \ce_m$ (respectively
$Y\in \ce_n$) such that $U=XV$ (respectively $U=VY$) if and only if $\kk (V)\subseteq\kk (U)$ (respectively $\rr (U)\subseteq\rr (V)$).
Now with these facts in mind, the notions in Definition \ref{def1100} and Definition~\ref{d24} can be extended to rectangular matrices.

\begin{df}\label{def_2.7}Let $A \in \matriz{n}{m}$ and $D, E \in \matriz{m}{n}$. 
\begin{enumerate}[{\rm (i)}]
\item The matrix  $A$ is said to be {\em left $(D,E)$-invertible}, if there exists $C \in \matriz{m}{n}$ 
	such that $CAD=D$ and $\kk(E) \subseteq \kk(C)$. Any matrix $C$ satisfying these conditions is
	said to be a left {\em $(D,E)$-inverse of $A$}.
\item The matrix $A$ is said to be {\em right $(D,E)$-invertible}, if there exists $B \in \matriz{m}{n}$ 
	such that $EAB=E$ and $\rr(B) \subseteq \rr(D)$. Any matrix $B$ satisfying these conditions is 
	said to be a right {\em $(D,E)$-inverse of $A$}.
\end{enumerate}
\end{df}

\indent The proofs of the following results are straightforward and they are left to the reader.

\begin{rema}\label{rema1200}\rm Let $A \in \matriz{n}{m}$ and $D, E \in \matriz{m}{n}$. The following statements hold\par
\begin{enumerate}[{\rm (i)}] 
\item The matrix $A$ is left $(D,E)$-invertible with a left inverse $C\in  \matriz{m}{n}$  if and only if
$A^*\in \matriz{m}{n}$ is right $(E^*, D^*)$-invertible and $C^*\in\matriz{n}{m}$ is a right  $(E^*, D^*)$-inverse of $A^*$.
\item The matrix $A$ is right $(D,E)$-invertible with a right inverse $B\in  \matriz{m}{n}$  if and only if
$A^*\in \matriz{m}{n}$ is left $(E^*, D^*)$-invertible and $B^*\in\matriz{m}{n}$ is a left  $(E^*, D^*)$-inverse of $A^*$.
\end{enumerate}
\noindent Consider $D', E' \in \matriz{m}{n}$ such that $\rr(D')=\rr(D)$ and $\kk(E')=\kk(E)$. The following statement holds.
\begin{enumerate}[{\rm (i)}] \setcounter{enumi}{2}
\item The matrix $A$ is left $(D, E)$-invertible if and only if it is left $(D', E')$-invertible. In addition, in this case $C\in\matriz{m}{n}$ 
is a left $(D, E)$-inverse of $A$ if and only if it is a  left $(D', E')$-inverse of $A$.
\item The matrix $A$ is right $(D, E)$-invertible if and only if it is right $(D', E')$-invertible. Moreover, in this case $B\in\matriz{m}{n}$ 
is a right $(D, E)$-inverse of $A$ if and only if it is a  right $(D', E')$-inverse of $A$.
\end{enumerate}
\end{rema}

\indent When the matrices $D$, $E\in \matriz{m}{n}$ in Definition \ref{def_2.7} coincide, the notions of left and right inverse along a matrix can 
be introduced.\par

\begin{df}\label{def_2.700}Let $A \in \matriz{n}{m}$ and $D\in \matriz{m}{n}$. 
\begin{enumerate}[{\rm (i)}]
\item The matrix  $A$ is said to be {\em left invertible along $D$}, if there exists $C \in \matriz{m}{n}$ 
	such that $CAD=D$ and $\kk(D) \subseteq \kk(C)$. Any matrix $C$ satisfying these conditions is
	said to be a {\em left inverse of $A$ along $D$}.
\item The matrix $A$ is said to be {\em right invertible along $D$}, if there exists $B \in \matriz{m}{n}$ 
	such that $DAB=D$ and $\rr(B) \subseteq \rr(D)$. Any matrix $B$ satisfying these conditions is 
	said to be a {\em right inverse of $A$ along $D$}.
\end{enumerate}
\end{df}

Note that similar results to the ones in Remark \ref{rema1200} for the case $D=E\in\matriz{m}{n}$ hold for left and right invertible
matrices along a matrix. The details are left to the reader.

Recall that given a ring $\rr$ and $a$, $b$, $c\in\rr$, in \cite[Definition 2.3]{KVC} the left and right annihilator $(b, c)$-inverses of the element $a$
were introduced. However, in the case of matrices, as under the conditions of \cite[Proposition 2.5]{KVC}, these notions coincide with the ones in 
Definition \ref{def_2.7}. 

%%%%%%%%%%%%%%%%%%%%%%%%%%%%%%%%%%%%%%%%%%%%%%%%%%%%

\section{Characterizations of the one-sided $(D,E)$-invertibility}

In first place matrices that satisfy Definition \ref{def_2.7} will be characterized.\par

\begin{thm}\label{thm301} Let $A \in \matriz{n}{m}$ and $D, E \in \matriz{m}{n}$. 
The following statements are equivalent.\par
\begin{enumerate}[{\rm (i)}]
\item The matrix $A$ is right $(D,E)$-invertible.
\item $\rr(E)=\rr(EAD)$.
\item $\rk(E)=\rk(EAD)$.
\item $\dim\kk(E)=\dim\kk(EAD)$.
\item   $\ce^n=\rr(AD) + \kk(E)$.
\item  $ \ce^m=\rr(D) + \kk(EA)$ and $\rk(EA)=\rk(E)$.
\end{enumerate}
\end{thm}
\begin{proof} 
In first place, it will be proved that statement (i) implies statement (ii). 
Assume that there exists a matrix $B \in \matriz{m}{n}$ such that
$EAB=E$ and $\rr(B) \subset \rr(D)$. Recall that the latter condition
is equivalent to the fact that there exists $M \in \ce_n$ such that $B = DM$. Therefore, 
$$
E=EAB=EADM.
$$
In particular, $\rr(E)=\rr(EAD)$.\par

Suppose that statement~(ii) holds. Thus, there exists 
$X\in \ce_n$ such
that $EADX=E$. To prove statement (i), it is enough to define $B=DX$.\par

Statements (ii), (iii), and (iv) are equivalent. In fact,
since $\rr(EAD)\subseteq \rr(E)$, statements (ii) and (iii) are equivalent. In addition, since
$\dim\kk(E)+\rk(E)=n=\dim\kk(EAD)+\rk(EAD)$, statements (iii) and (iv) are equivalent. 

Statements (i) and (v) are equivalent. In fact, according to what has been proved, if statement (i) holds, then there is $M \in \ce_n$ such that $E=EADM$. In particular, 
$\rr(I_n-ADM)\subseteq \kk(E)$. 
Since  any $\xn \in \ce^n$ can be written as $\xn = ADM \xn + (\xn-ADM\xn)$,
statement (v) holds. On the other hand, statement (v) implies statement (ii), 
since $\rr(E)=E(\ce^n) = E[\rr(AD) + \kk(E)] = \rr(EAD)$.\par

In this paragraph, it will be proved that statement (i) implies statement (vi). 
Assume that statement (i) holds. Then there exists a matrix  $B \in \matriz{m}{n}$  such that $E=EAB$ and $\rr(B)\subseteq\rr(D)$.
Since  any $\yn \in \ce^m$ can be written as $\yn=BA\yn+(\yn-BA\yn)$, 
the equality $\ce^m = \rr(D)+\kk(EA)$ is obtained. To prove the rank equality,
according to statement (iii), $\rk(E) = \rk(EAD) \leq \rk(EA) \leq \rk(E)$.\par

Finally, it will be proved that statement (vi) implies statement (iii). In fact, if statement (vi) holds, then 
$$
\rr(EA) = EA (\ce^m) = EA(\rr(D)+\kk(EA)) = \rr(EAD).
$$
However, $\rk(E)=\rk(EA)=\rk(EAD)$. %In particular, statement (iii) holds.
\end{proof}

\begin{thm}\label{thm302} Let $A \in \matriz{n}{m}$ and $D, E \in \matriz{m}{n}$. 
The following statements are equivalent.\par
\begin{enumerate}[{\rm (i)}]
\item The matrix $A$ is left $(D,E)$-invertible.
\item $\kk (D)=\kk (EAD)$.
\item $\dim \kk (D)=\dim\kk (EAD)$.
\item $\rk(D)=\rk(EAD)$.
\item $\kk(EA) \cap \rr(D)=0$.
\item $\rr(AD) \cap \kk(E)=0$ and  $\rk(D)=\rk(AD)$.
\end{enumerate}
\end{thm}

\begin{proof} Recall that according to Remark \ref{rema1200} (i),  $A$ is left-$(D,E)$-invertible 
if and only if $A^*\in \matriz{m}{n}$ is right $(E^*, D^*)$-invertible ($E^*$, $D^*\in\matriz{n}{m}$).
In addition, recall that 
\begin{align*}
&\kk(D)=\rr(D^*)^\perp,& &\kk(EAD)=\rr(D^*A^*E^*)^\perp.&\\
&\kk(EA)=\rr(A^*E^*)^\perp,& &\rr(D)=\kk(D^*)^\perp.& \\
&\rr(AD)=\kk(D^*A^*)^\perp,& &\kk(E)=\rr(E^*)^\perp.&\\
&\rk(D)=\rk(D^*),& &\rk(AD)=\rk(D^*A^*).&\\
\end{align*}
To conclude the proof, apply Theorem \ref{thm301} to $A^*$, $E^*$ and $D^*$, use the above identities and
note that since $\kk(EAD)\subseteq\kk(D)$, statememts (ii) and (iii) are equivalent. In addition, note that since 
$\dim\kk(D)+\rk(D)=n=\dim\kk(EAD)+\rk(EAD)$, statements (iii) and (iv) are equivalent. 
\end{proof}

Next given $D,E \in \ce_{m,n}$, left and right $(D, E)$-invertible matrices will be characterized using a particular map. 

\begin{thm}\label{th_3.5} 
Let $A \in \ce_{n,m}$ and $D,E \in \ce_{m,n}$. Let $\mathcal{X}$ be any subspace
of $\ce^n$ such that $\ce^n = \kk(E) \oplus \mathcal{X}$. Consider
$\phi:\rr(D) \to \mathcal{X}$ the map defined by 
$\phi(\xn)=P_{\mathcal{X},\kk(E)}(A\xn)$, for $\xn \in \rr(D)$.
The following statements hold.\par
\begin{enumerate}
\item[{\rm (i)}] The matrix $A$ is left $(D,E)$-invertible if and only if $\phi$ is injective.
\item[{\rm (ii)}] The matrix $A$ is right $(D,E)$-invertible if and only if $\phi$ is surjective.
\end{enumerate}
\end{thm}
\begin{proof} First statement (i) will be proved.
Observe that $\kk( \phi) = \rr(D) \cap \kk(EA)$. Thus, according to Theorem \ref{thm302}, 
$\kk(\phi)= 0$ if and only if $A$ is $(D,E)$-left invertible.

The assertion (ii) will be proved in this paragraph. 
Note that $\xn \in \rr(EAD)$ if and only if exists $\yn \in \ce^n$ such that
$$
\xn = E[P_{\mathcal{X},\kk(E)}(AD\yn) + P_{\kk(E),\mathcal{X}}(AD \yn)].
$$
Since $E[P_{\mathcal{X},\kk(E)}(AD\yn) + P_{\kk(E),\mathcal{X}}(AD \yn)] = E(\phi(D\yn))$,
the equality $\rr(EAD)=E(\rr(\phi))$ is obtained. In addition, the linear mapping
$f:\rr(\phi) \to \ce^m$ given by $f(\xn)=E\xn$ is injective (because $\rr(\phi) \subseteq
\mathcal{X}$ and $\mathcal{X} \oplus \kk(E)=\ce^n$), 
therefore, $\dim \rr(\phi) = \dim  E(\rr(\phi)) 
= \rk(EAD)$, and thus, $\phi$ is surjective (which is equivalent to $\dim \rr(\phi) = \dim \mathcal{X}$)  if and only if $\rk(EAD)=\rk(E)$.
However, according to Theorem 3.1, 
this latter condition is equivalent to  the fact that $A$ is right $(D, E)$-invertible. 
\end{proof} 

\indent In the following theorem matrices satisfying simultaneously 
Theorem \ref{thm301} and Theorem~\ref{thm302} will be studied.\par

\begin{thm}\label{thm3033} Let $A \in \ce_{n,m}$ and $D,E \in \ce_{m,n}$. The following statements are equivalent.
\begin{enumerate}[{\rm (i)}]
\item $A$ is left and right $(D, E)$-invertible.
\item $\rr(EAD)=\rr(E)$ and $\kk(EAD)=\kk(D)$.
\item $\rr(AD) \oplus \kk(E) = \ce^n$ and $\rk(D)=\rk(AD)$.
\item $\rr(D) \oplus \kk(EA) = \ce^m$ and $\rk(E)=\rk(EA)$.
\item The map $\phi:\rr(D) \to \mathcal{X}$ 
defined in Theorem \ref{th_3.5} is bijective.
\end{enumerate}
\noindent Futhermore, in this case, $\rk(E)=\rk(D)$.
\end{thm}
\begin{proof} Apply Theorem \ref{thm301}, Theorem \ref{thm302} and Theorem \ref{th_3.5}. 
Note also that $\rk(E)=\rk(D)$. Actually, this equality 
can be derived from the fact that $n = \rk(EAD)+ \dim \kk (EAD)=\rk(E)+\dim \kk(D)$.
\end{proof}

Next the left and right $(D, E)$-inverses of a matrix $A$ satisfying Theorem \ref{thm3033} will be characterized.

\begin{pro}\label{pro3034} Let $A \in \ce_{n,m}$ and $D,E \in \ce_{m,n}$ be such that $A$ is both left and right 
$(D, E)$-invertible. Then, there exist only one left $(D, E)$-inverse of $A$ and only one right $(D, E)$-inverse of $A$. Moreover,
these inverses coincide with the 
unique matrix $R\in\ce_{m, n}$ satisfying 
$$
\kk(R) = \kk(E), \qquad R \yn = f^{-1}(\yn), \ \ \forall \yn \in \rr(AD),
$$  
where $f:\rr(D) \to \rr(AD)$ is the isomorphism defined by $f(\xn) = A  \xn$.
\end{pro}

\begin{proof}
Consider $C\in\ce_{m, n}$ a left $(D, E)$-inverse of $A$. Then, $CAD=D$ and $\kk(E)\subseteq\kk(C)$. 
Let $f:\rr(D) \to \rr(AD)$ and $g:\rr(AD) \to \rr(D)$ be given by 
$f(\xn)=A\xn$ and $g(\yn)=C\yn$. If $\xn \in \rr(D)$, then $gf(\xn)=CA\xn$
and $\xn = D \un$ for some $\un \in \ce^n$. From $CAD=D$, it is obtained that
$gf(\xn)=\xn$. In a similar way, $fg = I_{\rr(AD)}$ can be proved, and therefore, 
$g=f^{-1}$.

In this paragraph it  will be proved that $\kk(C)=\kk(E)$. Since  $\kk(E) \subseteq \kk(C)$
is already known, it is enough to prove the opposite inclusion. Let $\xn \in \kk(C)$, 
by Theorem 3.4 (iii), $\xn$ can be written as $\xn = AD \yn + \wn$, where $\yn \in \ce^n$
and $\wn \in \kk(E)$. Now, $\on = C\xn = CAD \yn + C \wn = D \yn$ because
$\wn \in \kk(E) \subseteq \kk(C)$. Finally, $\xn = AD\yn + \wn = \wn \in  \kk(E)$.

\indent Now consider $B\in \ce_{m,n}$ a right $(D, E)$-inverse of $A$. In particular, $EAB=E$ and $\rr(B)\subseteq \rr(D)$.
Let $x\in\kk(E)$. Then, $B(x)\in\rr(D)\cap\kk(EA)=0$ (Theorem \ref{thm3033} (iv)). Thus, $\kk(E)\subseteq \kk(B)$.
The inclusion $\kk(B) \subseteq \kk(E)$ is evident from $EAB=E$. Therefore, $\kk(B)=\kk(E)$.\par
\indent Let $h:\rr(AD) \to \rr(D)$ and $k: \rr(D) \to \rr(E)$
defined by $h(\yn)=B\yn$ and $k(\yn)=E\yn$. The mapping $k$ is an isomorphism
because it is simple to prove in view of Theorem 3.4 that $\kk(k)=0$.
Furthermore, $EAB=E$ leads to $kfh=k$, and using that $k$ is an isomorphism, 
$fh=I_{\rr(AD)}$, i.e., $h=f^{-1}$.
\end{proof}

\indent Now the relationship between the notions introduced in Definitions \ref{def_2.700} will be studied. To this end, in first place a characterization of left invertibility along a matrix will be given.

\begin{thm}\label{thm3038} Let $A \in \ce_{n,m}$ and $D\in \ce_{m,n}$. The following statements are equivalent.\par
\begin{enumerate}[{\rm (i)}]
\item $A$ is right invertible along $D$.
\item $\rr(D)=\rr(DAD)$.
\item $\rk(D)=\rk(DAD)$.
\item $\ce^n=\kk(D)\oplus\rr(AD)$.
\item $\ce^m=\rr(D)\oplus\kk(DA)$.
\end{enumerate}
\end{thm}
\begin{proof} Suppose that statement (i) holds. Then, according to Theorem \ref{thm301} applied to the case $D=E$,
statements (ii) and (iii) hold, $\rk(DA)=\rk(D)$,   $\ce^n=\kk(D)+\rr(AD)$ and  $\ce^m=\rr(D)+\kk(DA)$. Now, since
\begin{equation*}
\begin{split}
n&=\rk(AD)+\dim \kk(D)-\dim [\rr(AD)\cap\kk(D)]\\
&=\rk(D)+\dim \kk(D)-\dim [\rr(AD)\cap\kk(D)],\\
\end{split}
\end{equation*}
\noindent $ \rr(AD)\cap\kk(D)=0$ and statement (iv) holds.\par
\indent Similarly, since
\begin{equation*}
\begin{split}
m&=\rk(D)+\dim \kk(DA)-\dim [\rr(D)\cap\kk(DA)]\\
&=\rk(DA)+\dim \kk(DA)-\dim [\rr(D)\cap\kk(DA)],\\
\end{split}
\end{equation*}
\noindent $ \rr(D)\cap\kk(DA)=0$ and statement (v) holds.\par
\indent On the other hand, note that statement (ii) (respectively (iii), (iv), (v)) implies 
statement (ii) (respectively (iii), (iv), (v)) of  Theorem \ref{thm301} applied to the case $D=E$. For statement (v), note also that
since $\ce^m=\rr(D)\oplus\kk(DA)$, 
the equality $\rk(D)=\rk(DA)$ can be obtained.
\end{proof}

\indent It is possible to obtain similar statements for right invertible elements along a matrix, however,
as the following theorem shows, left and right inverse along a matrix are equivalent notions.

\begin{thm}\label{thm3039} Let $A \in \ce_{n,m}$ and $D\in \ce_{m,n}$. The following statements are equivalent.\par
\begin{enumerate}[{\rm (i)}]
\item $A$ is right invertible along $D$.
\item $A$ is left invertible along $D$.
\item $\kk(D)=\kk(DAD)$.
\item $\dim \kk(D)=\dim \kk(DAD)$.
\item $AD\in\ce_n$ is group invertible and $\dim \kk(D)=\dim\kk(AD)$.
\item $DA\in \ce_m$ is group invertible and $\rk(DA)=\rk (D)$.
\item The map $\phi:\rr(D) \to \mathcal{X}$ 
defined in Theorem \ref{th_3.5} for the case $D=E$ is bijective.
\end{enumerate}
\end{thm}
\begin{proof}According to Theorem \ref{thm302} applied to the case $D=E$, statements (ii),
(iii) and (iv) are equivalent.
In addition, note that statement (iv) is equivalent to Theorem \ref{thm3038} (iii). In particular, statements (i) and (ii) are equivalent. \par

\indent Suppose that statement (i) holds. Then according to Theorem \ref{thm3038} (iv), $\ce^n=\kk(D)\oplus\rr(AD)$. Moreover, according to Theorem \ref{thm302} (v),
$\rk(D)=\rk(AD)$. However, the latter identity is equivalent to $\rr(D)=\rr (AD)$, which in turn is equivalent to $\kk(D)=\kk(AD)$. In particular, $\ce^n=\kk(AD)\oplus\rr(AD)$, i.e.,
$AD$ is group invertible,  and $\dim \kk(D)=\dim\kk(AD)$. 

On the other hand, if statement (v) holds, then   $\ce^n=\kk(AD)\oplus\rr(AD)$ and $\kk(D)= \kk(AD)$ ($\kk(AD)\subseteq \kk(D)$).
Consequently, Theorem \ref{thm3038} (iv) holds.\par

\indent The equivalence between statements (i) and (vi) can be proved a similar argument, using in particular Theorem \ref{thm3038} (v) and Theorem \ref{thm301} (vi).\par
\indent Since statements (i) and (ii) are equivalent, according to Theorem \ref{thm3033}, statement (i) and (vii) are equivalent.
\end{proof}

In the following corollary the left and the right inverses of a matrix $A\in \ce_{n,m}$ that is left or right invertible along $D\in  \ce_{m,n}$ will be presented.

\begin{cor}\label{cor31000} Let $A \in \ce_{n,m}$ and $D\in \ce_{m,n}$ such that $A$ is  left or right invertible along $D$. 
Then, there exists only one left inverse of $A$ along $D$ and only one right inverse of $A$ along $D$. Moreover,
these inverses coincide with the matrix $R\in\ce_{m, n}$ satisfying
$$
\kk(R)=\kk(D), \qquad R \yn = f^{-1}(\yn), \ \ \forall \yn \in \rr(AD),
$$  
where $f: \rr(D) \to \rr(AD)$ is given by $f(\xn)=A\xn$.
\end{cor}
\begin{proof} Apply Theorem \ref{thm3039}  and Proposition \ref{pro3034}.
\end{proof}

\indent Now the existence of left  and right $(D, E)$-inverses will be studied. To this end the sets of left and right $(D, E)$-invertible matrices will be characterized. First of all some notation will be given.

\indent Consider  $D$, $E\in \matriz{m}{n}$. Let $(\matriz{n}{m})^{\parallel D, E}_{left}$ and    $(\matriz{n}{m})^{\parallel D, E}_{right}$
be the sets of left and right $(D, E)$-invertible matrices, respectively, i.e.,
\begin{align*}
&(\matriz{n}{m})^{\parallel D, E}_{left}=\{ A\in\matriz{n}{m}\colon A \hbox{ is left } (D, E)\hbox{-invertible}\},\\
&(\matriz{n}{m})^{\parallel D, E}_{right}=\{ A\in\matriz{n}{m}\colon A \hbox{ is right } (D, E)\hbox{-invertible}\}.\\
\end{align*}

\indent When $D=E\in \matriz{m}{n}$, the sets of left and right invertible matrices along  $D$,  
$(\matriz{n}{m})^{\parallel D}_{left}$ and    $(\matriz{n}{m})^{\parallel D}_{right}$ respectively, are introduced.

\begin{equation*}
\begin{split}
&(\matriz{n}{m})^{\parallel D}_{left}=(\matriz{n}{m})^{\parallel D, D}_{left}=\{ A\in\matriz{n}{m}\colon A \hbox{ is left invertible along } D\},\\
&(\matriz{n}{m})^{\parallel D}_{right}=(\matriz{n}{m})^{\parallel D, D}_{right}=\{ A\in\matriz{n}{m}\colon A \hbox{ is right invertible along } D\}.\\
\end{split}
\end{equation*}

Next conditions under which the sets $(\matriz{n}{m})^{\parallel D, E}_{left}$ and $(\matriz{n}{m})^{\parallel D, E}_{right}$ 
are non empty will be given.

\begin{thm}\label{thm303} Let $D, E \in \matriz{m}{n}$. The following statements are equivalent.\par
\begin{enumerate}[{\rm (i)}]
\item $(\matriz{n}{m})^{\parallel D, E}_{left}\ne\emptyset$. 
\item $\rk(D)\le\rk(E)$.
\item $\dim \kk(E)\le\dim \kk(D)$.
\item $\dim \kk(E)+\rk(D)\le n$.
\end{enumerate}
\end{thm}

\begin{proof} 
Here, It will be proved that statement (i) implies statement (ii). 
Suppose that there exists $A \in \matriz{n}{m}$ and $C\in \matriz{m}{n}$ such that
$C$ is a right $(D, E)$-inverse of $A$. In particular, $CAD=D$ and $\kk(E)\subseteq \kk(C)$.  
Thus, $\rk(D)\le\rk(C)$ and $\dim \kk(E)\le\dim \kk(C)$. As a result,
$\rk(D)\le\rk(C)\le \rk(E)$.

Suppose that statement (iv) holds. 
Let $r=\rk(E)$ and $s=\rk(D)$. Let $X = [\xn_1 \cdots \xn_n] \in \ce_n$ 
and $Y = [\yn_1 \cdots \yn_m] \in \ce_m$ be two nonsingular matrices
such that the last $n-r$ columns of $X$ span $\kk(E)$ and the first
$s$ columns of $Y$ span $\rr(D)$. Define
$$
A = X \left[ \begin{array}{cc} I_s & 0 \\ 0 & 0 \end{array} \right] Y^{-1} \in \ce_{n,m}, 
\qquad
C = Y \left[ \begin{array}{cc} I_s & 0 \\ 0 & 0 \end{array} \right] X^{-1} \in \ce_{m,n}. 
$$
If $\xn \in \kk(E)$, then $\xn = \sum_{i=r+1}^n \alpha_i \xn_i$ for some scalars $\alpha_i$.
Since $s= \rk(D) \le \rk(E)=r$, the vector $\vn = [0 \cdots 0 \ \alpha_{r+1} \cdots \alpha_n]^T
\in \ce^{n-s}$ can be defined (the superscript $T$ means the transposition). Hence
$\xn = X \left[ \begin{smallmatrix} \on \\ \vn \end{smallmatrix} \right]$, and thus, 
$C\xn = \on$.
If $\yn \in \rr(D)$, then exists $\wn \in \ce^s$ such that $\yn = 
Y \left[ \begin{smallmatrix} \wn \\ \on \end{smallmatrix} \right]$ 
(because the first $s$ columns of $Y$ span $\rr(D)$). Now, it is trivial 
to prove $CA\yn = \yn$, which implies $CAD=D$ since $\yn \in \rr(D)$ is arbitrary.
Hence (i) holds.

The remaining equivalences are clear.
\end{proof}

\begin{thm}\label{thm304} Let $A \in \matriz{n}{m}$ and $D, E \in \matriz{m}{n}$. The following statements are equivalent.\par
\begin{enumerate}[{\rm (i)}]
\item $(\matriz{n}{m})^{\parallel D, E}_{right}\ne\emptyset$. 
\item $\rk(E)\le\rk(D)$.
\item $\dim \kk(D)\le \dim \kk(E)$.
\item $n\le\dim \kk(E)+\rk(D)$.
\end{enumerate}
\end{thm}
\begin{proof} Recall that according to Remark \ref{rema1200} (ii),  $A$ is right $(D,E)$-invertible 
if and only if $A^*\in \matriz{m}{n}$ is left $(E^*, D^*)$-invertible ($E^*$, $D^*\in\matriz{n}{m}$).
Consequently, statement (i) is equivalent to $(\matriz{m}{n})^{\parallel E^*, D^*}_{left}\ne\emptyset$,
which in turn is equivalent to $\rk(E^*)\le\rk(D^*)$. However, the latter inequality coincides with statement (ii).

The remaining equivalences are clear.
\end{proof}

Next the case of the left and right inverses along a matrix will be considered.\par

\begin{cor}\label{cor305}Let $D\in \matriz{m}{n}$. Then,\par
$$
(\matriz{n}{m})^{\parallel D}_{left}=(\matriz{n}{m})^{\parallel D}_{right}\neq\emptyset.
$$
\end{cor}
\begin{proof} Apply Theorem \ref{thm3039} and Theorem \ref{thm303} or Theorem \ref{thm304} for the case $D=E$.
\end{proof}

\indent Next  the results in Corollary \ref{cor305} will be extended to the case $\dim \kk(E)+\rk(D)=n$ ($D$, $E\in \matriz{m}{n}$).
Note that this condition is equivalent to $\rk(D)=\rk(E)$, which in turn is equivalent to $\dim \kk(D)=\dim \kk(E)$, which is also equivalent 
to $\dim \kk(D)+\rk(E)=n$.

\begin{cor}\label{cor3050} Let $D$, $E\in\matriz{m}{n}$. The following statements are equivalent.
\begin{enumerate}[{\rm (i)}]
\item $\rk(E)=\rk(D)$.
\item $(\matriz{n}{m})^{\parallel D, E}_{left}\ne\emptyset$ and $(\matriz{n}{m})^{\parallel D, E}_{right}\ne\emptyset$.
\item $(\matriz{n}{m})^{\parallel D, E}_{left}=(\matriz{n}{m})^{\parallel D, E}_{right}\neq\emptyset$.
\end{enumerate}
\end{cor}
\begin{proof}To prove the equivalence between statements (i) and (ii), apply Theorem \ref{thm303} and Theorem \ref{thm304}.\par

\indent Suppose that statment (i) holds and consider $A\in (\matriz{n}{m})^{\parallel D, E}_{right}$. According to Theorem \ref{thm301} (iii), $\rk(D)=\rk (EAD)$.
Thus, $\dim\kk(D)=\dim\kk(EAD)$. Consequently, according to Theorem \ref{thm302} (iii), $A\in (\matriz{n}{m})^{\parallel D, E}_{left}$. A similar argument,
using in particular that $\dim \kk(D)=\dim \kk(E)$, proves that $(\matriz{n}{m})^{\parallel D, E}_{left}\subseteq (\matriz{n}{m})^{\parallel D, E}_{right}$. On the other hand,
if statement (iii) holds, then consider $A\in (\matriz{n}{m})^{\parallel D, E}_{left}=(\matriz{n}{m})^{\parallel D, E}_{right}$. According to Theorem \ref{thm301}
and Theorem  \ref{thm302}, $\rk(E)=\rk(EAD)$ and $\kk(D)=\kk(EAD)$. Therefore, $\dim \kk(D)+\rk(E)=n$.
\end{proof}

Now the case $\rk(D)\neq\rk(E)$ will be presented.

\begin{cor}\label{cor3051} Let $D$, $E\in\matriz{m}{n}$ such that $\rk(D)\neq\rk(E)$. Then, the following statements hold.
\begin{enumerate}[{\rm (i)}]
\item If $\rk(D)<\rk(E)$, then $(\matriz{n}{m})^{\parallel D, E}_{left}\ne\emptyset$ and $(\matriz{n}{m})^{\parallel D, E}_{right}=\emptyset$.
\item  If $\rk (E)<\rk D)$, then $(\matriz{n}{m})^{\parallel D, E}_{right}\ne\emptyset$ and $(\matriz{n}{m})^{\parallel D, E}_{left}=\emptyset$.
\end{enumerate}
\end{cor}
\begin{proof} Apply Theorem \ref{thm303} and Theorem \ref{thm304}. 
\end{proof}

\indent Next representations of the sets $(\matriz{n}{m})^{\parallel D, E}_{left}$ and $(\matriz{n}{m})^{\parallel D, E}_{right}$
will be given. However, in first place some notation needs to be introduced.\par

Let $D$, $E\in\matriz{m}{n}$ and consider $A\in (\matriz{n}{m})^{\parallel D, E}_{left}$. Then, \it the set of all left $(D, E)$-inverses of $A$ \rm will be denoted by $\mathcal{I}(A)_{left}^{\parallel D, E}$.
Similarly, when $A\in (\matriz{n}{m})^{\parallel D, E}_{right}$, $\mathcal{I}(A)_{right}^{\parallel D, E}$ will stand for \it the set of all
right $(D, E)$-inverses of $A$. \rm

In the following theorems representations of $(\matriz{n}{m})^{\parallel D, E}_{left}$,  $(\matriz{n}{m})^{\parallel D, E}_{right}$, 
$\mathcal{I}(A)_{left}^{\parallel D, E}$ and $\mathcal{I}(A)_{right}^{\parallel D, E}$ will be given.

\begin{thm}\label{thm30000} Let $D, E\in\matriz{m}{n}$ be such that
$(\ce_{n,m})_{left}^{\| D,E} \neq \emptyset$. Let $s=\rk(D)\le \rk(E)=r$. Then
\begin{enumerate}[{\rm (i)}]
\item $A \in (\ce_{n,m})_{left}^{\| D,E}$ if and only if there exist two nonsingular matrices
$X \in \ce_n$ and $Y \in \ce_m$ such that 
\begin{equation}\label{axy-1}
A = X \kbordermatrix{ & s & m-s \\ s & A_1 & * \\ r-s & 0 & * \\ n-r & 0 & *} Y^{-1},
\end{equation}
where the last $n-r$ columns of $X$ are a basis of $\kk(E)$, the first $s$ columns
of $Y$ are a basis of $\rr(D)$, and $A_1$ is nonsingular.
\item Under the conditions in statement {\rm (i)}, $C \in \mathcal{I}(A)^{\parallel D, E}_{left}$
if and only if 
\begin{equation}\label{cteorema}
C = Y \kbordermatrix{ & s & r-s & n-r \\ s & A_1^{-1} & * & 0 \\ m-s & 0 & * & 0 } X^{-1}.
\end{equation}
\end{enumerate}
\end{thm}
\begin{proof}
Assume that $A \in (\ce_{n,m})_{left}^{\| D,E}$.
Let $X \in \ce_n$ be any nonsingular matrix such that the last $n-r$ columns span 
$\kk(E)$. Let $Y \in \ce_m$ be any nonsingular matrix such that the first $s$ columns 
span $\rr(D)$. Let us decompose matrix $A$ as follows:
\begin{equation}\label{x-1ay}
A = X \kbordermatrix{ & s & m-s \\ r & B_1 & * \\ n-r & B_2 & *} Y^{-1}.
\end{equation}
Observe that if $\yn_i$ is the $i$-th column of $Y$, then $A \yn_i \in \rr(AD)$
for $i=1, \ldots, s$, because the first $s$ columns of $Y$ span $\rr(D)$.
From \eqref{x-1ay}, it is obtained that $B_2=0$, because $\rr(AD)\cap\kk(E)=0$
(Theorem~\ref{thm302} (vi)).

Let $\mathcal{X}$ be the subspace spanned by the first $r$ columns of $X$ 
(this subspace satisfies $\mathcal{X} \oplus \kk(E)=\ce^n$). It is
evident that the matrix of the linear mapping $\phi: \rr(D) \to \mathcal{X}$
defined in Theorem~\ref{th_3.5} respect to the considered basis of $\rr(D)$ and
$\mathcal{X}$ is $B_1$. By Theorem~\ref{th_3.5} (i), $\phi$ is injective, 
thus, $\mathcal{X}$ can be decomposed as $\mathcal{X} = \phi(\rr(D)) \oplus \mathcal{X}_1$, 
and without loss of generality, the first $n-r$ columns of $X$ can be rearranged
so that the first $s$ are a basis of $\phi(\rr(D))$. In this way, the decomposition
written in the statement (i) is obtained. 

Assume that $A \in \ce_{n,m}$ is decomposed as in \eqref{axy-1}. As in the previous paragraph,
let $\mathcal{X}$ be the subspace spanned by the first $r$ columns of $X$ and 
consider the mapping $\phi:\rr(D) \to \mathcal{X}$ defined in Theorem~\ref{th_3.5}. The first
$s$ columns of $Y$ is a basis of $\rr(D)$ and the first $r$ columns of $X$ is a basis
of $\mathcal{X}$. The matrix of $\phi$ respect to the aforementioned basis
is $\left[ \begin{smallmatrix} A_1 \\ 0 \end{smallmatrix} \right]$. Since $A_1$
is nonsingular, the mapping $\phi$ is injective, and according to Theorem~\ref{th_3.5} (i),
$A$ is left $(D,E)$-invertible.

Now, it will be proved statement (ii). Let $\yn_j$ be the $j$-th column of $Y$
and $\xn_i$ be the $i$-th column of $X$.

Assume that $C \in \ce_{m,n}$ is a left $(D,E)$-inverse of $A$. Let us decompose
$C$ as follows:
\begin{equation}\label{cyx-1}
C = Y \kbordermatrix{& s & r-s & n - r \\ s & C_1 & C_2 & C_3 \\ m - s & C_4 & C_5 & C_6} X^{-1}.
\end{equation}
Since $\yn_j = Y \left[ \begin{smallmatrix} \en_j \\ \on \end{smallmatrix} \right]$, 
where $\{ \en_1, \ldots, \en_s \}$ is the standard basis of  $\ce^n$, from the above decomposition
and \eqref{axy-1}, it is obtained that $CA \yn_j = 
\left[ \begin{smallmatrix} C_1 A_1 \en_j \\ C_4 A_1 \en_j \end{smallmatrix} \right]$, for any
$j=1, \ldots, s$. From $CAD=D$ (this is true because $C$ is a left $(D,E)$-inverse of $A$), it 
follows that $C_1A_1 \en_j=\en_j$ and $C_4 A_1 \en_j = \on$, for any $j=1, \ldots, s$.
Therefore $C_1=A_1^{-1}$ and $C_4=0$. Having in mind that the last $n-r$ columns of $X$
is a basis of $\kk(E)$ and $\kk(E) \subseteq \kk(C)$, the decomposition \eqref{cyx-1}
yields that $C_3$ and $C_6$ are zero matrices.

Assume that $C \in \ce_{m,n}$ is written as in \eqref{cteorema}.
The following two relations will be proved: $CAD=D$ and $\kk(E) \subseteq \kk(C)$. To prove that
$CAD=D$, it is enough to check that $CA \yn_j = \yn_j$, for $j=1, \ldots, s$ (because
$\rr(D)$ is spanned by $\yn_1, \ldots, \yn_s$),	and this trivially follows from
\eqref{axy-1}, \eqref{cteorema}, and 
$\yn_j = Y \left[ \begin{smallmatrix} \en_j \\ \on \end{smallmatrix} \right]$, 
where $\{ \en_j \}_{j=1}^s$ is the standard basis of $\ce^s$. To prove that $\kk(E) \subseteq \kk(C)$,
it is enough to prove that $C \xn_i=\on$, for $i= r+1, \ldots, n$ (because $\kk(E)$ is spanned
by the last $n-r$ columns of $X$). This is trivial in view of \eqref{cteorema}.
\end{proof}

\begin{thm}\label{thm30001} 
Let $D$, $E\in\matriz{m}{n}$ be such that $(\ce_{n,m})_{right}^{\| D,E } \neq \emptyset$. Let $r=\rk(E)\le\rk(D)=s$. Then
\begin{enumerate}[{\rm (i)}]
\item $A \in (\ce_{n,m})_{right}^{\| D,E }$ if and only if there exist two nonsingular
matrices $X \in \ce_n$ and $Y \in \ce_m$ such that
\begin{equation} \label{teorema2a}
A = X \kbordermatrix{ & s-r & r & m-s \\ r & 0 & A_2 & * \\ n-r & * & * & *}Y^{-1},
\end{equation}
where  the last $n-r$ columns of $X$ are a basis of $\kk(E)$, 
the first $s$ columns of $Y$ are a basis of $\rr(D)$, and $A_2$ is nonsingular.
\item Under the conditions in statement {\rm (i)}, $B \in 
\mathcal{I}(A)^{\parallel D, E}_{right}$ if and only if
\begin{equation} \label{teorema2b}
B = Y \kbordermatrix{& r & n-r \\ s-r & * & * \\ r & A_2^{-1} & 0 \\ m-s & 0 & 0} X^{-1}.
\end{equation}
\end{enumerate}
\end{thm}
\begin{proof} 
Assume that $A \in (\ce_{n,m})_{right}^{\| D,E }$.
Let $X \in \ce_n$ be any nonsingular matrix such that the last $n-r$ columns span $\kk(E)$
and let $\mathcal{X}$ be the subspace spanned by the first $r$ columns of $X$. 
The mapping $\phi\colon\rr(D)\to\mathcal{X}$ defined in Theorem \ref{th_3.5} is surjective,
and therefore, $s=\dim \rr(D) = \dim \kk(\phi) + \dim \mathcal{X} = \dim \kk(\phi) + r$. 
In Theorem \ref{th_3.5}, it was proved that $\kk(\phi)=\rr(D)\cap\kk(EA)$.
Let $Y \in \ce_m$ be any nonsigular matrix such that the first $s-r$ columns of $Y$
are a basis of $\rr(D) \cap \kk(EA)$
and the first $s$ columns of $Y$ are a basis of $\rr(D)$.  Decompose
$$
A = X \kbordermatrix{ & s-r & r & m-s \\ r & A_1 & A_2 & A_3 \\ n-r & A_4 & A_5 & A_6}Y^{-1}.
$$
Let $\yn_j$ be the $j$-th column of $Y$. For $j=1, \ldots, s-r$, it is obtained that $EA \yn_j=\on$,
because the first $s-r$ columns of $Y$ belong to $\kk(EA)$, in other words, 
$A \yn_j \in \kk(E)$, and thus, $A_1$ is a zero matrix.
The matrix of $\phi$ respect the considered basis of $\rr(D)$ and $\mathcal{X}$ is 
$[A_1 \ A_2] = [0 \ A_2]$. Since $\phi:\rr(D) \to \mathcal{X}$ is surjective, 
$r = \dim \mathcal{X} = \rk [0 \ A_2] = \rk(A_2)$. By recalling that $A_2$ is an $r \times r$
matrix, $A_2$ is nonsingular.

Assume that $A \in \ce_{n,m}$ is decomposed as in \eqref{teorema2a}. As in the previous
paragraph, let $\mathcal{X}$ be the subspace spanned by the first $r$ columns of $X$.
The matrix of the mapping $\phi: \rr(D) \to \mathcal{X}$ respect the considered basis is
$$
\kbordermatrix{ & s-r & r \\ r & 0 & A_2}.
$$
Since the rank of this latter matrix is $r$ (because $A_2 \in \ce_r$ is nonsingular)
and $\dim \mathcal{X}=r$, the mapping $\phi$ is surjective. According to Theorem \ref{th_3.5} (ii), $A\in (\ce_{n,m})_{right}^{\| D,E }$.

Assume that $B \in \ce_{m,n}$ is a right $(D,E)$-inverse of $A$, i.e., 
$EAB=B$ and $\rr(B) \subseteq \rr(D)$. Decompose 
$$
B= Y \kbordermatrix{ & r & n-r \\ s-r & B_1 & B_2 \\ r & B_3 & B_4 \\ m-s & B_5 & B_6} X^{-1}.
$$
By the condition $\rr(B) \subseteq \rr(D)$, it is obtained that $B \xn \in \rr(D)$ for 
any $\xn \in \ce^n$. Therefore, 
$$
Y \left[ \begin{array}{cc} B_1 & B_2 \\ B_3 & B_4 \\ B_5 & B_6 \end{array} \right]
\left[ \begin{array}{c} \xn_1 \\ \xn_2 \end{array} \right] \in \rr(D), \quad \forall \ 
(\xn_1, \xn_2 ) \in \ce^r \times \ce^{n-r}.
$$
Recall that the first $s$ columns of $Y$ span $\rr(D)$, and thus, 
$B_5 \xn_1 + B_6 \xn_2 = \on$, for all $(\xn_1, \xn_2 ) \in \ce^r \times \ce^{n-r}$, 
which implies that $B_5$ and $B_6$ are zero matrices. The equality
$EAB=E$ is equivalent to $\rr(AB-I_n) \subseteq \kk(E)$. But
$$
AB-I_n = 
X \left[ \begin{array}{ccc} 0 & A_2 & * \\ * & * & * \end{array} \right]
\left[ \begin{array}{cc} B_1 & B_2 \\ B_3 & B_4 \\ 0 & 0 \end{array} \right]X^{-1} - XX^{-1}
= X \left[ \begin{array}{cc} A_2 B_3 - I_r & A_2 B_4 \\ * & * \end{array} \right] X^{-1}.
$$
Therefore, $X \left[ \begin{smallmatrix} A_2 B_3-I_r & A_2 B_4 \\ * & * \end{smallmatrix} \right]
\left[ \begin{smallmatrix} \xn_1 \\ \xn_2 \end{smallmatrix} \right] 
\in \kk(E)$, for all $(\xn_1, \xn_2) \in \ce^r \times \ce^{n-r}$. Recall that 
the last $n-r$ columns of $X$ span $\kk(E)$, which implies that $A_2B_3=I_r$ and 
$B_4=0$.

Assume that $B \in \ce_{m,n}$ is written as in \eqref{teorema2b}. It will be 
proved now that $\rr(B) \subseteq \rr(D)$. If $\xn \in \ce^n$, then
$$
B \xn = Y \left[ \begin{array}{cc} * & * \\ A_1^{-1} & 0 \\ 0 & 0 \end{array} \right]X^{-1} \xn
\in \rr(D),
$$
because the first $s$ columns of $Y$ belong to $\rr(D)$.
Now it will be proved that $EAB=B$. Since
$$
AB-I_n = X \left[ \begin{array}{ccc} 0 & A_2 & * \\ * & * & * \end{array} \right]
\left[ \begin{array}{cc} * & * \\ A_2^{-1} & 0 \\ 0 & 0 \end{array} \right]X^{-1}-XX^{-1}
= X \left[ \begin{array}{cc} 0 & 0 \\ * & * \end{array} \right]X^{-1},
$$
it is obtained that $(AB-I_n) \xn \in \kk(E)$, for any $\xn \in \ce^n$, because the last $n-r$ columns
of $X$ belong to $\kk(E)$. Thus, $\rr(AB-I_n) \subseteq \kk(E)$, which is equivalent to
$EAB=E$.
\end{proof}

Next the case $\rk(E)=\rk(D)$ will be studied, i.e., when $(\ce_{n,m})_{left}^{\| D,E} = (\ce_{n,m})_{right}^{\| D,E}\neq\emptyset$  ($D, E \in \matriz{m}{n}$). 
Compare with  Proposition~\ref{pro3034}.

\begin{cor}\label{cor30005} 
Let $D, E \in \matriz{m}{n}$ be such that $\rk(E)=\rk(D)=r$.
Let $X \in \ce_n$ be any nonsingular matrix such that its last $n-r$ columns span $\kk(E)$ and 
$Y \in \ce_m$ be any nonsingular matrix such that its first $r$ columns span $\rr(D)$. If
$A \in \ce_{n,m}$ is written as 
$$
A = X \kbordermatrix{ & r & n-r \\ r & A_1 & * \\ m-r & * & * } Y^{-1},
$$
then $A \in (\ce_{n,m})_{left}^{\| D,E} = (\ce_{n,m})_{right}^{\| D,E}$ if and only if 
$A_1$ is nonsingular. Furthermore, 
under the above equivalence, the unique left and the unique right $(D, E)$-inverse of $A$ is 
$$
Y  \left[ \begin{array}{cc} A_1^{-1} & 0 \\ 0 & 0 \end{array} \right]X^{-1}.
$$
\end{cor}
\begin{proof} 
Let $\mathcal{X}$ be the subspace spanned by the first $r$ columns of $X$. 
The matrix of the mapping $\phi: \rr(D) \to \mathcal{X}$ defined in Theorem~\ref{th_3.5}
respect the considered basis is $A_1$. Therefore, according to Theorem~\ref{thm3033}, 
the matrix $A_1$ is nonsingular if and only if $A$ is left and right $(D,E)$-invertible. 
The proof of Theorem~\ref{thm30000} (ii) shows that the unique left and right $(D,E)$-inverse 
of $A$ as the form of the statement of this corollary.
\end{proof}

\begin{rema}\label{nota1}
\rm 
Let $D, E \in \matriz{m}{n}$ be such that $\rk(E)=\rk(D)$ and $A \in (\ce_{n,m})_{left}^{\| D,E}
= (\ce_{n,m})_{right}^{\| D,E}$.
Let $\ce^n = \kk(E) \oplus \mathcal{X}$ be any decomposition. From Corollary \ref{cor30005}, 
the unique left and right $(D,E)$-inverse of $A$, say $R$, satisfies
$$
\kk(R) = \kk(E), \qquad R \yn = \phi^{-1}(\yn), \ \ \forall \yn \in \mathcal{X}. 
$$
\end{rema}

Next the case of left and right inverses along a fixed matrix will be considered.

\begin{cor}\label{cor30200} 
Let $D\in\matriz{m}{n}$ and $r =  \rk(D)$.
Let $X \in \ce_n$ be any nonsingular matrix such that its last columns span $\kk(D)$ and 
$Y \in \ce_m$ be any nonsingular matrix such that its first columns span $\rr(D)$. If
$A \in \ce_{n,m}$ is written as 
$$
A = X \left[ \begin{array}{cc} A_1 & * \\ * & * \end{array} \right] Y^{-1},
\quad A_1 \in \ce_r, 
$$
then $A \in (\ce_{n,m})_{left}^{\| D} = (\ce_{n,m})_{right}^{\| D}$ if and only if 
$A_1$ is nonsingular. Furthermore, 
under the above equivalence, the unique left and the unique right inverse of $A$ along $D$ is 
$$
Y  \left[ \begin{array}{cc} A_1^{-1} & 0 \\ 0 & 0 \end{array} \right]X^{-1}.
$$
\end{cor}
\begin{proof} Apply Corollary \ref{cor305} and Corollary \ref{cor30005}.
\end{proof}
\indent  In the following theorem the sets 
 $\mathcal{I}(A)_{left}^{\parallel D, E}$ and $\mathcal{I}(A)_{right}^{\parallel D, E}$ will be represented 
using the Moore-Penrose inverse.  

\begin{thm}\label{5L1} Let $D, E \in \matriz{m}{n}$. The following statements hold.
\begin{enumerate}[{\rm (i)}]
\item If $A\in (\matriz{n}{m})^{\parallel D, E}_{right}$, then 
	$$
	\mathcal{I}(A)_{right}^{\parallel D, E}=\left\{ D \left[ (EAD)^\dag E + 
	\left(I_n -(EAD)^\dag EAD \right)Z\right] : Z \in \ce_n\right\}.
	$$
\item If $A\in (\matriz{n}{m})^{\parallel D, E}_{left}$, then 
$$
\mathcal{I}(A)_{left}^{\parallel D, E}=\left\{ \left[ D(EAD)^\dag + Z(I_m - EAD(EAD)^\dag) \right]E : Z \in \ce_m\right\}.
$$
\end{enumerate}
\end{thm}

\begin{proof} Consider $A\in (\matriz{n}{m})^{\parallel D, E}_{right}$. If $B \in \matriz{m}{n}$ satisfies
$EAB = E$ and $\rr(B) \subset \rr(D)$, then according to the proof of Theorem \ref{thm301}, there exists a matrix $M$
such that $B=DM$ and $EADM=E$. Notice that the general solution of the equation $EADX=E$ is
$$
X = (EAD)^\dag E + (I_n-(EAD)^\dag EAD)Z,
$$
where $Z \in \matriz{n}{n}$ is arbitrary (see \cite[Theorem 2]{penrose}).
Hence $M = (EAD)^\dag E + (I_n-(EAD)^\dag EAD)Z$,
for some $Z \in \matriz{n}{n}$. Therefore, $B=D(EAD)^\dag E + D(I_n-(EAD)^\dag EAD)Z$.
Thus,
$$(\matriz{n}{m})^{\parallel D, E}_{right} \subseteq
\{ D(EAD)^\dag E + D(I_n-(EAD)^\dag EAD)Z : Z \in \matriz{n}{n} \}.
$$
\indent To prove the opposite inclusion, let $Z \in \matriz{n}{n}$ be arbitrary and
consider
$$
Y = D \left[(EAD)^\dag E + (I_n-(EAD)^\dag EAD)Z \right].
$$
It is evident that $\rr(Y) \subset \rr(D)$. Furthermore, according to Theorem \ref{thm301} (ii),
$\rr(E)=\rr(EAD)$. Now
$$
EAY= EAD \left[(EAD)^\dag E + (I_n-(EAD)^\dag EAD)Z \right] = EAD (EAD)^\dag E = E.
$$
\noindent In fact, if $\en$ is any column of $E$, then $\en \in \rr(E) = \rr(EAD)$ and
since $EAD(EAD)^\dag$ is the orthogonal projection onto $\rr(EAD)$, 
$EAD(EAD)^\dag \en = \en$.

\indent To prove statement (ii), apply Remark \ref{rema1200} (i) and what has been proved. 
\end{proof}

Let $D,E \in \ce_{m,n}$. If $A \in \ce_{n,m}$ is left (respectively right) $(D, E)$-invertible,
the case in which $\mathcal{I}(A)_{left}^{\parallel D, E}$ (respectively $\mathcal{I}(A)_{right}^{\parallel D, E}$) is a singleton will be studied.  

\begin{thm}\label{thm33000} Let $A \in \ce_{n,m}$ and $D,E \in \ce_{m,n}$. The following statements are equivalent.\par
\begin{enumerate}[{\rm (i)}]
\item The matrix $A$ has a unique left $(D, E)$-inverse.
\item  The matrix $A$ has a unique right $(D, E)$-inverse.
\item $\rk(D)=\rk(E)=\rk(EAD)$.
\end{enumerate}
\noindent Furthermore, in this case $\mathcal{I}(A)_{left}^{\parallel D, E}=\mathcal{I}(A)_{right}^{\parallel D, E}=\{D(EAD)^\dag E\}$.
\end{thm}
\begin{proof} Note that according to Theorem \ref{thm302} (iv) and Theorem \ref{thm30000} (ii), statement (i) and statement (iii) are equivalent. 

To prove the equivalence between statements (ii) and (iii), apply Theorem \ref{thm301} (iii) and Theorem \ref{thm30001} (ii).

Now, according to Proposition \ref{pro3034}, the unique left and the unique right $(D, E)$-inverse of $A$ coincide. 
To conclude the proof, notice that according to Theorem \ref{5L1}, $D(EAD)^\dag E\in \mathcal{I}(A)_{left}^{\parallel D, E}\cap\mathcal{I}(A)_{right}^{\parallel D, E}$.
\end{proof}

Due to Theorem \ref{thm33000}, another representation of the left and the right inverses along a matrix can be given.

\begin{cor}\label{cor37000} Let $A \in \ce_{n,m}$ and $D\in \ce_{m,n}$. The matrix  $A$ is  left or right invertible along $D$ if and only if 
$\rk(D)=\rk(DAD)$. Moreover, in this case the unique left inverse of $A$ along $D$ and the unique right inverse of $A$ along $D$ coincide with the matrix $D(DAD)^\dag D$.
\end{cor}
\begin{proof}
Apply Theorem \ref{thm3039}, Corollary \ref{cor31000}  and Theorem \ref{thm33000}.
\end{proof}

It is known that a nonzero matrix can be expressed as the product of a matrix of 
full column rank and a matrix of full row rank. This factorization is known as 
a full-rank factorization and these factorizations turn out to be a 
powerful tool in the study of generalized inverses. Recall that given $H\in\ce_{m,n}$
such that $\rk(H)=r>0$, the matrix $H$ is said to have a \it full rank factorization,  \rm 
if there exist $F\in\ce_{m,r}$ and $G\in\ce_{r,n}$ such that $H=FG$. Such a factorization always exists but it is not unique (\cite[Theorem 2]{PO}).
Moreover, $F^\dag F=I_r= GG^\dag$ (\cite[Theorem 1]{PO}). Note in particular that $\rk(F)=\rk(G)=\rk(H)$. To learn more results on this topic, see \cite{ben,PO}.
In the following theorem, given $D$, $E\in\ce_{m,n}$, matrices $A\in\ce_{n,m}$ that are left or right $(D, E)$-invertible will be characterized using a full rank factorization. 
In addition, the sets $\mathcal{I}(A)_{left}^{\parallel D, E}$ and $\mathcal{I}(A)_{right}^{\parallel D, E}$ will be represented using a full rank 
factorization and the Moore-Penrose inverse.

\begin{thm}\label{th_3.4}
Let $A \in \ce_{n,m}$ and $D,E \in \ce_{m,n}$. Consider
$D=D_1D_2$ and $E=E_1E_2$ two full rank factorizations of $D$ and $E$, respectively.
\begin{enumerate}[{\rm (i)}]
\item The matrix $A$ has a right $(D,E)$-inverse if and only if $\rk(E_2)=\rk(E_2AD_1)$. In addition,
$$
\mathcal{I}(A)_{right}^{\parallel D, E}=\{D_1 \left[ (E_2AD_1)^\dag E_2 + 
(I_r-(E_2AD_1)^\dag E_2AD_1)Y\right]\colon Y \in \ce_{r,n} \},
$$ 
\noindent where $r=\rk(D)$.
\item The matrix $A$ has a left $(D,E)$-inverse if and only if $\rk(D_1)=\rk(E_2AD_1)$. Moreover,
$$
\mathcal{I}(A)_{left}^{\parallel D, E}=\{ \left[ D_1(E_2AD_1)^\dag + 
Y(I_s-(E_2AD_1)( E_2AD_1)^\dag) \right]E_2: Y \in \ce_{m,s} \},
$$
\noindent  where $s=\rk(E)$.
\end{enumerate}
\end{thm}
\begin{proof} Let $r=\rk(D)$ and $s=\rk(E)$. According to \cite[Theorem 1]{PO}, $D_2D_2^\dag = I_r$, $E_2E_2^\dag = I_s$, $D_1^\dag D_1=I_r$ and 
$E_1^\dag E_1=I_s$. In addition, since
$$
\rk(EAD) = \rk(E_1E_2AD_1D_2) \leq \rk(E_2 A D_1) \leq \rk(E_1^\dag E_1 E_2 A D_1 D_2 D_2^\dag)
\leq \rk(EAD),
$$
$\rk(EAD)=\rk(E_2AD_1)$. According to Theorem \ref{thm301} (iii), $A$ 
has a right $(D,E)$-inverse if and only if $\rk(E_2)=\rk(E)=\rk(EAD)=\rk(E_2AD_1)$. 

Let $A \in(\matriz{n}{m})^{\parallel D, E}_{right}$ and consider $B\in\mathcal{I}(A)_{right}^{\parallel D, E}$, i.e., $EAB=E$ and $\rr(B) \subseteq \rr(D)$.
Since $E_1E_2AB=E_1E_2$, multiplying by $E_1^\dag$ on the left hand side of this equation,  $E_2AB=E_2$. In addition, since
$\rr(B) \subseteq \rr(D)$, there exists $M \in \ce_n$ such that $B=DM$. Therefore, 
$E_2AD_1(D_2M)=E_2$, and according to \cite[Theorem 2]{penrose}, 
there exists $Y \in \ce_{r,n}$ such that
$D_2M= (E_2AD_1)^\dag E_2 + \left[I_r-(E_2AD_1)^\dag (E_2AD_1)\right]Y$. Thus,
$B=DM=D_1D_2M$ implies 
$$
B = D_1 \left[ 
(E_2AD_1)^\dag E_2 + \left(I_r-(E_2AD_1)^\dag E_2AD_1\right)Y\right].
$$
\indent Now suppose that  $B \in \ce_{m,n}$  has this form. Observe that 
$B=D_1Z$ for some matrix $Z$. Thus, $B=D_1D_2D_2^\dag Z = DD_2^\dag Z$
implies that $\rr(B) \subseteq \rr(D)$. In addition, 
$$
EAB = E_1E_2AD_1 \left[ 
(E_2AD_1)^\dag E_2 + \left(I_r-(E_2AD_1)^\dag E_2AD_1\right)Y \right]
= E_1E_2AD_1 (E_2AD_1)^\dag E_2.
$$
Since $\rk(E_2AD_1)=\rk(E_2)$ and $\rr(E_2AD_1) \subseteq \rr(E_2)$, $\rr(E_2AD_1)=\rr(E_2)$.
Therefore,
$E_2AD_1 (E_2AD_1)^\dag = P_{\rr(E_2AD_1)} = P_{\rr(E_2)}$, which 
implies $E_2AD_1 (E_2AD_1)^\dag E_2=E_2$. Thus,
$EAB=E_1E_2=E$. In particular, $B\in \mathcal{I}(A)_{right}^{\parallel D, E}$.

To prove statement (ii), apply Remark \ref{rema1200} (i) and what has been proved.
\end{proof}

\indent Next two particular cases will be derived from Theorem \ref{th_3.4}.

\begin{cor}\label{cor3800} Let $D,E \in \ce_{m,n}$ and consider $D=D_1D_2$ and $E=E_1E_2$ two full rank factorizations of $D$ and $E$, respectively.
The matrix $A \in \ce_{n,m}$ is left and right $(D, E)$-invertible if and only if $\rk(E_2)=\rk(D_1)=\rk(E_2AD_1)$. Moreover, in this case, 
$$
\mathcal{I}(A)_{left}^{\parallel D, E}=\mathcal{I}(A)_{right}^{\parallel D, E}=\{D_1(E_2AD_1)^{-1} E_2\}.
$$
\end{cor}
\begin{proof} The first statement can be derived from Theorem \ref{th_3.4}. \par

Note that $\rk(D)=\rk(D_1)$, $\rk(E)=\rk(E_2)$ and $\rk(EAD)=\rk(E_2AD_1)$ (see the proof of Theorem \ref{th_3.4}).
Then, according to Theorem \ref{thm33000}, $\mathcal{I}(A)_{left}^{\parallel D, E}=\mathcal{I}(A)_{right}^{\parallel D, E}$ is a singleton.
In addition, according to Theorem \ref{th_3.4}, $D_1(E_2AD_1)^\dag E_2\in \mathcal{I}(A)_{left}^{\parallel D, E}\cap\mathcal{I}(A)_{right}^{\parallel D, E}$.
Thus, 
$$
\mathcal{I}(A)_{left}^{\parallel D, E}=\mathcal{I}(A)_{right}^{\parallel D, E}=\{D_1(E_2AD_1)^\dag E_2\}.
$$
\noindent  However, since $\rk(E_2)=r=\rk(D_1)$ and $(E_2AD_1)^\dag \in\ce_r$ is such that $\rk(E_2AD_1)=r$,
$(E_2AD_1)^\dag=(E_2AD_1)^{-1}$.
\end{proof}

\indent To end this section, the case of left and right inverses along a  matrix will be presented.

\begin{cor}\label{cor3900} Let $D\in \ce_{m,n}$ and consider $D=D_1D_2$ a full rank factorizations of $D$.
The matrix  $A \in \ce_{n,m}$
is left or right invertible along $D$ if and only if $\rk(D_1)=\rk(D_2AD_1)$. Moreover, in this case the unique left inverse of $A$ along $D$ and the unique right inverse of $A$ along $D$ 
coincide with $D_1(D_2AD_1)^{-1} D_2$.
\end{cor}
\begin{proof} Apply Theorem \ref{thm3039} and Corollary \ref{cor3800}.
\end{proof}

%%%%%%%%%%%%%%%%%%%%%%%%%%%%%%%%%%%%%%%%%%%%%%%%%%

\section{The $(D, E)$-inverse of arbitrary matrices}

\indent First of all the $(b, c)$-inverse will be extended to rectangular matrices.
Compare with Definition~\ref{def1} and recall the observation
before Definition \ref{def_2.7}. 

\begin{df}\label{def3035} Let $A \in \ce_{n,m}$ and $D,E \in \ce_{m,n}$. The matrix $A$ is said to be $(D, E)$-{\rm invertible}, 
if there exist a matrix $X\in \ce_{m,n}$ such that the following conditions hold.
$$
XAD=D,\hskip.3truecm EAX=E,\hskip.3truecm \rr(X)\subseteq \rr(D),\hskip.3truecm \kk(E)\subseteq\kk(X).
$$
\end{df}

Under the same conditions as in Definition \ref{def3035}, note that $\rr(X)\subseteq \rr(D)$ (respectively $\kk(E)\subseteq\kk(X)$)
is equivalent to $\rr(X)= \rr(D)$  (respectively $\kk(E)\subseteq\kk(X)$). In the following theorem, it will be proved that the $(D, E)$-inverse of a matrix $A$ is unique, if it exists.\par

\begin{thm}\label{thm3036} Let $A \in \ce_{n,m}$ and $D,E \in \ce_{m,n}$. The following statements are equivalent.\par
\begin{enumerate}[{\rm (i)}]
\item The $(D, E)$-inverse of the matrix $A$ exists.
\item The matrix $A$ is both left and right $(D, E)$-invertible.
\end{enumerate}
\noindent Furthermore, in this case,  the $(D, E)$-inverse of the matrix $A$ is unique.
\end{thm}
\begin{proof} It is enough to prove that statement (ii) implies statement (i). To this end, apply Proposition  \ref{pro3034}.
Proposition  \ref{pro3034} also proves that there is only one $(D, E)$-inverse of $A$, when it exists.
\end{proof}

\indent According to Theorem \ref{thm3036}, if the matrix $A\in \ce_{n,m}$ has a $(D, E)$-inverse ($D,E \in \ce_{m,n}$), then
it will be denoted by $A^{\parallel (D, E)}$. In addition, note that according to Definition \ref{def1100} and \cite[Corollary 3.7]{KVC}, when the matrices $A$, $D$ and $E$
are square, Definition \ref{def3035} reduces to  the $(b, c)$-inverse (\cite[Definition 1.3]{Drazin}, i.e., Definition \ref{def1}). In the following remark some basic results on this inverse that can be derived from
what has been proved in sections 2 and  3 will be collected.

\begin{rema}\label{rema40000}\rm Let $D, E \in \ce_{m,n}$ and consider $A\in \ce_{n,m}$.
\begin{enumerate}[{\rm (i)}]
\item The matrix $A$ is $(D, E)$-invertible if and only if $A^*\in\ce_{m,n}$ is $(E^*, D^*)$-invertible ($E^*$, $D^*\in\ce_{n,m}$).
Moreover, in this case $(A^*)^{\parallel (E^*, D^*)}=(A^{\parallel (D, E)})^*$. Apply Theorem \ref{thm3036} and Remark \ref{rema1200} (i)-(ii).
\item Let $D', E' \in \ce_{m,n}$ be such that $\rr(D)=\rr(D')$ and $\kk(E)=\kk(E')$. Necessary and sufficient for $A^{\parallel (D, E)}$ to exist is that
$A^{\parallel (D', E')}$ exists. Furthermore, in this case, $A^{\parallel (D', E')}=A^{\parallel (D, E)}$. Apply Theorem \ref{thm3036} and Remark \ref{rema1200} (iii)-(iv).
\item Theorem \ref{thm3033} and Theorem \ref{thm33000} characterize matrices $A$ such that $A^{\parallel (D, E)}$ exists. 
\item When $A$ is $(D, E)$-invertible, $A^{\parallel (D, E)}$ can be represented as 
in Propostion~\ref{pro3034}, Corollary~\ref{cor30005} and Remark~\ref{nota1}.
\end{enumerate} 
\end{rema}

Although some results have been presented in connection to left and right $(D, E)$-inverses ($D,E \in \ce_{m,n}$), they deserve to 
be considered again for the $(D, E)$-inverse. Recall that according to \cite[Remark 2.4]{Drazin} (see also \cite[Theorem 2.2]{Drazin}), when the marices $A$, $D$, $E$ are square,
$A$ is $(D, E)$-invertible if and only if $\rk(D)=\rk(EAD)=\rk(E)$. In the following theorem this result will be extended to arbitrary matrices.

\begin{thm}\label{thm60000} Let $A \in \ce_{n,m}$ and $D,E \in \ce_{m,n}$. The following statements are equivalent.\par
\begin{enumerate}[{\rm (i)}]
\item The $(D, E)$-inverse of $A$ exists.
\item $\rk(D)=\rk(E)=\rk(EAD)$.
\end{enumerate}
\noindent Furthermore, in this case $A^{\parallel (D, E)}=D(EAD)^\dag E$.
\end{thm}
\begin{proof} Apply Theorem \ref{thm3036}, Proposition \ref{pro3034} and Theorem \ref{thm33000}.
\end{proof}

Now a corollary will be derived from Theorem \ref{thm60000}.

\begin{cor}\label{cor_rank}
Let $A \in \matriz{n}{m}$ and $D,E \in \matriz{m}{n}$ be such that $\ADE$ exists.
Then $\rk(AD)=\rk(EA)=\rk(E)=\rk(D)$.
\end{cor}
\begin {proof} Apply Theorem \ref{thm60000} and Theorem \ref{thm3033}.
\end{proof}

\begin{rema}\rm Let $A \in \matriz{n}{m}$ and $D,E \in \matriz{m}{n}$.
Note that the condition in Corollary \ref{cor_rank} does not imply that $\ADE$ exists.
In fact, consider
$$
A = \mat{0}{1}{1}{0}, \qquad D = E = \mat{1}{0}{0}{0}.
$$
Since $EAD = 0$ and $\rk(D) = \rk(E) = 1$, according to  Theorem \ref{thm60000}, $\ADE$ 
does not exist. However
$$
AD = \mat{0}{0}{1}{0}, \qquad EA = \mat{0}{1}{0}{0},
$$
which lead to $\rk(AD) = 1$ and $\rk(EA) = 1$. 
\end{rema}
 
In the following theorem the $(D, E)$-inverse will be characterized using full rank factorizations.

\begin{thm}\label{thm61000} Let $D,E \in \ce_{m,n}$ and consider $D=D_1D_2$ and $E=E_1E_2$ two full rank factorizations of $D$ and $E$, respectively.
The matrix $A \in \ce_{n,m}$ is  $(D, E)$-invertible if and only if $\rk(E_2)=\rk(D_1)=\rk(E_2AD_1)$. Moreover, in this case, $A^{\parallel (D, E)}=D_1(E_2AD_1)^{-1} E_2$.
\end{thm}
\begin{proof} Apply Theorem \ref{thm3036} and Corollary \ref{cor3800}.
\end{proof}

In the following corollaries two particular cases will be derived from Theoerm \ref{thm61000}.

\begin{cor}\label{cor_rankbis} Let $A \in \matriz{n}{m}$ and $D,E \in \matriz{m}{n}$ be such that
exists $\ADE$ exists. The following statements hold.
\begin{enumerate}
\item[{\rm (i)}] If the columns of $D$ are linearly independent, then
$\ADE = D(AD)^{-1}$.
\item[{\rm (ii)}] If the rows of $E$ are linearly independent, then
$\ADE = (EA)^{-1}E$.
\end{enumerate}
\end{cor}
\begin{proof}
If $D$ has full column rank, then $D=DI_n$ is a full rank factorization of $D$. 
Since $\ADE$ exists, according to Theorem \ref{thm60000}, $\rk(E)=\rk(D)$. Consequently, $E$ has full column rank and 
$E=EI_n$ is a full rank factorization of $E$. Then, according to Theorem \ref{thm61000}, 
$\ADE=D(AD)^{-1}$. 

Apply a similar argument to prove statement (ii), using in particular the full rank factorizatons $E=I_m E$ and $D=I_mD$ and 
Theorem \ref{thm61000}.
\end{proof}

\begin{cor}\label{c_rank_one} Let $A \in \matriz{n}{m}$ and $D, E \in \matriz{m}{n}$ be such that $\rk(D) =\rk(E)= 1$. 
If $D = \dn_1 \dn_2^*$ and $E = \en_1 \en_2^*$ \rm(\it$\dn_1, \en_1\in \matriz{m}{1}$, 
$\dn_2, \en_2\in \matriz{n}{1}$\rm) \it are full rank factorizations of $D$ and $E$, respectively, 
then $\ADE$ exists if and only if $\en_2^* A \dn_1 \neq 0$. Moreover, in
this case,
$$
\ADE = \frac{1}{\en_2^* A \dn_1} \dn_1 \en_2^*.
$$
\end{cor}
\begin{proof} According to Theorem \ref{thm61000},
the $(D, E)$-inverse of the matrix $A$ exists if and only if
$\rk(\dn_1) = \rk(\en_2^*) = \rk(\en_2^* A \dn_1)$. Therefore,
$\ADE$ exists if and only if $\rk(\en_2^* A \dn_1)=1$. But observe that
$\en_2^* A \dn_1$ is a complex number. Thus,  the first part of the theorem has been proved. 
The expression of $\ADE$  also follows from Theorem \ref{thm61000}.	
\end{proof}

Next an application of Theorem \ref{thm61000} will lead to representations of several
generalized inverses in terms of a full rank representation.
The explicit expression for $A^\dagger$ is attributed to C.C. MacDufee by Ben-Israel and Greville
in \cite{ben}. Ben-Israel and Greville report that around 1959, MacDufee was the first to point
out that a full-rank factorization of $A$ leads to the mentioned formula.

\begin{cor} \label{cor61050} Let
$A \in \ce_n$ and consider  a full rank factorization $A=FG$. Then, the following statements hold.
\begin{enumerate}[{\rm (i)}]
\item $A^\dag = G^* (F^* A G^*)^{-1}F^*$.
\item  
$A$ is group invertible if and only if $GF$ is nonsingular; in this case,
$A^\# = F(GF)^{-2}G$ {\rm{(}}see {\rm \cite{cline}}{\rm{)}}.
\item $A$ is core invertible if and only if $GF$ is nonsingular; in this case, $\core{A}=F(GF)^{-1}F^\dag$.
\item $A$ is dual core invertible if and only if $GF$ is nonsingular; in this case, $\cored{A}=G^\dag(GF)^{-1}G$.
\item $A_{M,N}^\dag = N^{-1}G^* (F^*MAN^{-1}G^*)^{-1}F^*M$, where $M$ and $N$ are 
nonsingular and positive.
\end{enumerate}
\end{cor}
\begin{proof}
Observe that $A^*=G^*F^*$ is a full rank factorization of $A^*$.\par

\indent Recall that according to \cite[p. 1912]{Drazin}, $A$ is Moore-Penrose invertible if and only if $A$ is $(A^*, A^*)$-invertible and in this case $A^\dag=A^{\parallel (A^*, A^*)}$.
Now apply Theorem \ref{thm61000} with $D=E=A^*=G^*F^*$. 

According to \cite[p. 1910]{Drazin}, $A$ is group invertible if and only if $A$ is $(A, A)$-invertible and in this case $A^\#=A^{\parallel (A, A)}$. 
According to Theorem \ref{thm61000}, $A$ is group invertible if and only if $\rk(G)=\rk(F)=\rk(GFGF)$. Thus, 
if $r=\rk(A)$, then $GF\in\ce_r$ is invertible.  To prove the formula representing $A^\#$, apply Theorem \ref{thm61000} with $D=E=A=FG$.

Recall that according to \cite[p. 684]{bt},  $A$ is core 
invertible if and only if it is group invertible. Thus, according to what has been proved, the first part of statement (iii) holds.
In addition, according to \cite[Theorem 4.4]{rdd} (i), necessary and sufficient for $A$ to be core invertible is that $A$ is $(A, A^*)$-invertible and in this case
$\core{A}=A^{\parallel (A, A^*)}$. Now, apply  Theorem \ref{thm61000} with $D=A=FG$ and $E=A^*=G^*F^*$. Observe however first that if $r=\rk(A)$, then
$F^*F\in \ce_r$ is such that $\rk(F^*F)=\rk(F)=r$, and recall that $(F^*F)^{-1}F^*=F^\dag$ (\cite[Theorem 1]{PO}).  Then,
$$
\core{A} = A^{\parallel (A, A^*)} = F(F^*FGF)^{-1} F^* = F(GF)^{-1}(F^*F)^{-1}F^*=F(GF)^{-1}F^\dag.
$$

Note that $A$ is dual core dual invertible if and only if $A^*$ is core invertible and
$\cored{A}=[\core{(A^*)}]^*$. Thus, the first part of statement (iv) can be derived from what has been proved and
$$
\cored{A}=[{(G^*(F^*G^*)^{-1}(G^*)^\dag)}]^*=G^\dag(GF)^{-1}G.
$$

According to \cite[Theorem  3.2]{bb},  $A$ is weighted
Moore-Penrose invertible with weights $M$ and $N$ if and only if
$A$ is invertible along $N^{-1} A^* M$. In addition, according to \cite[Definition~6.1]{Drazin},
this is equivalent to the fact that $A$ is $(N^{-1} A^* M, N^{-1} A^* M)$-invertible. Now apply
Theorem \ref{thm61000} with $D=E=N^{-1} A^* M$ and consider the full rank factorization $N^{-1} A^* M=(N^{-1}G^*)(F^*M)$. 
\end{proof}

To end this section,
the set of all matrices $A\in \ce_{n,m}$ such that they are $(D, E)$-invertible will be studied ($D$, $E \in \ce_{m,n}$).
First some notation need  to be introduced.

\indent  Given $D$, $E \in \ce_{m,n}$, let $\matriz{n}{m}^{\parallel D, E}$ denote the set of all matrices $A\in \ce_{n,m}$ such that the $(D, E)$-inverse of $A$ exists, i.e.,
$$
\matriz{n}{m}^{\parallel D, E}=\{A\in \ce_{n,m}\colon A^{\parallel (D, E)}\hbox{ exists }\}.
$$

\begin{thm}\label{thm4001} Let $D$, $E \in \ce_{m,n}$. 
Necessary and sufficient for $\matriz{n}{m}^{\parallel D, E}\neq\emptyset$
is that $\rk(E)=\rk(D)$. Moreover, in this case, 
$\matriz{n}{m}^{\parallel D, E} = 
(\matriz{n}{m})_{left}^{\parallel D, E}=(\matriz{n}{m})_{right}^{\parallel D, E}$. 
\end{thm}
\begin{proof} Apply Theorem \ref{thm3036} and Corollary \ref{cor3050}. % and Corollary \ref{cor30005}.
\end{proof}

Observe that Corollary \ref{cor30005} gives an explicit representation of 
$\matriz{n}{m}^{\parallel D, E}$.

%%%%%%%%%%%%%%%%%%%%%%%%%%%%%%%%%

\section{Invertible matrices along a fixed matrix} 
Now the case $D=E$ will be considered. Compare with Definition \ref{def1000} and recall the observation
before Definition \ref{def_2.7}. 

\begin{df}\label{def3037} Let $A \in \ce_{n,m}$ and $D\in \ce_{m,n}$. The matrix $A$ is said to be {\rm invertible along the matrix} $D$, 
if there exists a matrix $X\in \ce_{m,n}$ such that the following conditions hold.
$$
XAD=D=DAX,\qquad \rr(X)\subseteq \rr(D),\qquad \kk(D)\subseteq\kk(X).
$$
\end{df}

According to Theorem \ref{thm3036}, if the inverse of the matrix $A\in \ce_{n,m}$ along the matrix $D\in \ce_{m,n}$ exists, then
it is unique and  it will be denoted by $A^{\parallel D}$.  In addition,  when $A$ and $D$ are square matrices, since according to \cite[Proposition 6.1]{Drazin}, the inverse of $A$ along $D$ coincides with 
the $(D, D)$-inverse of $A$, Definition \ref{def3037} reduces to the notion of the inverse along an element in a ring of square matrices (\cite[Definition 4]{Mary}, i.e., Definition \ref{def1000}).
Moreover, under the same conditions as in Definition \ref{def3037},
note that  $\rr(X)\subseteq \rr(D)$ (respectively $\kk(D)\subseteq\kk(X)$) is equivalent to $\rr(X)=\rr(D)$ (respectively $\kk(D)=\kk(X)$).
Now characterizations of the inverse along a matrix will be given.

\begin{thm}\label{thm4100} Let $A \in \ce_{n,m}$ and $D \in \ce_{m,n}$. The following statements are equivalent.
\end{thm}
\begin{enumerate}[{\rm (i)}]
\item The matrix $A$ is left invertible along $D$.
\item The matrix $A$ is right invertible along $D$.
\item The matrix $A$ is  invertible along $D$.
\end{enumerate}
\begin{proof}Apply Theorem \ref{thm3036} and Theorem \ref{thm3039}.
\end{proof}

In the following remark several results on this inverse that can be derived from
what has been proved in sections 2, 3 and 4 will be collected.

\begin{rema}\label{rema41000}\rm Let $D\in \ce_{m,n}$ and consider $A\in \ce_{n,m}$.
\begin{enumerate}[{\rm (i)}]
\item The matrix $A$ is invertible along $D$ if and only if $A^*\in\ce_{m,n}$ is invertible along $D^*\in\ce_{n,m}$.
Moreover, in this case $(A^*)^{\parallel  D^*}=(A^{\parallel D)})^*$. Apply Remark \ref{rema40000} (i) to the case $E=D$.
\item Let $D' \in \ce_{m,n}$ be such that $\rr(D)=\rr(D')$ and $\kk(D)=\kk(D')$. Necessary and sufficient for $A^{\parallel D}$ to exists is that
$A^{\parallel D'}$ exists. Furthermore, in this case, $A^{\parallel D'}=A^{\parallel D}$. Apply Remark \ref{rema40000} (ii) to the case $E=D$.
\item Theorem \ref{thm3038} and Theorem \ref{thm3039} characterize matrices $A$ such that $A^{\parallel D}$ exists. Compare Theorem \ref{thm3039} (v)-(vi)
with \cite[Theorem 7]{Mary} and \cite[Theorem 2.1]{MP}.
\item When $A$ is invertible along $D$, $A^{\parallel D}$ can be represented as in Corollary \ref{cor31000} and Corollary \ref{cor30200}. 
\item According to Remark \ref{nota1} applied to the case $D=E$, another representation of $A^{\parallel D}$, when it exists, is the following. 
In the decomposition $\ce^n=\kk(D)\oplus\mathcal{X}$, it holds that
$$
\kk(A^{\| D}) = \kk(D), \qquad A^{\| D} \yn = \phi^{-1}(\yn), \ \yn \in \mathcal{X},
$$
%In this decompositions,  $A^{\parallel D}$ can be represented as 
%$$
%A^{\parallel D}\mid_{\kk(D)}=0,\hskip.3cm A^{\parallel D}\mid_{\kk(D)^\perp}=\phi^{-1}\colon \kk(D)^\perp\to {\rr(D)},
%$$
\noindent where $\phi\colon\rr(D)\to\mathcal{X}$ is the map of Theorem \ref{th_3.5}.
\item According to Corollary \ref{cor_rank} applied to the case $D=E$, if $A^{\parallel D}$ exists, then  $\rk(AD)=\rk(DA)=\rk(D)$.
\end{enumerate} 
\end{rema}

Some results, however, deserve to be presented separately.

\begin{cor}\label{cor70000} Let $A \in \ce_{n,m}$ and $D \in \ce_{m,n}$. The following statements are equivalent.\par
\begin{enumerate}[{\rm (i)}]
\item $A$ is invertible along $D$.
\item $\rk(D)=\rk(DAD)$.
\end{enumerate}
\noindent Furthermore, in this case $A^{\parallel D}=D(DAD)^\dag D$.
\end{cor}
\begin{proof} Apply Theorem \ref{thm60000} for the case $D=E$.
\end{proof}

\begin{cor}\label{cor71000} Let $D\in \ce_{m,n}$ and consider $D=D_1D_2$ a full rank factorizations of $D$.
The matrix $A \in \ce_{n,m}$ is  invertible along $D$ if and only if $\rk(D_1)=\rk(D_2AD_1)$. Moreover, in this case, $A^{\parallel D}=D_1(D_2AD_1)^{-1} D_2$.
\end{cor}
\begin{proof} Apply Theorem \ref{thm61000} for the case $D=E$.
\end{proof}

\begin{cor}\label{cor_rankbister} Let $A \in \matriz{n}{m}$ and $D \in \matriz{m}{n}$. 
\begin{enumerate}
\item[{\rm (i)}] Suppose that  the columns of $D$ are linearly independent. Necessary and sufficient for $A^{\parallel D}$
to exists is that $AD\in\ce_n$ is nonsingular. Moreover, in this case, $A^{\parallel D} = D(AD)^{-1}$.
\item[{\rm (ii)}] Suppose that the rows of $D$ are linearly independent. The inverse of $A$ along $D$ exists if and only if
$DA\in\ce_m$ is invertible. Moreover, in this case, $A^{\parallel D} = (DA)^{-1}D$.
\end{enumerate}
\end{cor}
\begin{proof} Suppose that  the columns of $D$ are linearly independent. According to Theorem \ref{thm4100} and Theorem \ref{thm3039} (v),
the characterization of the inverse of $A$ along $D$ holds. To prove the formula that represents $A^{\parallel D}$, apply Corollary \ref{cor_rankbis} (i).

To prove  statement (ii), apply a similar argument to the one used to prove statement (i), using in particular Theorem \ref{thm3039} (vi) and Corollary \ref{cor_rankbis} (ii).
\end{proof}
\begin{cor}\label{c_rank_oneter} 
Let $A \in \matriz{n}{m}$ and $D \in \matriz{m}{n}$.
If $\rk(D)= 1$, 
%If $D = \dn_1 \dn_2^*$ \rm (\it $\dn_1\in \matriz{m}{1}$ and 
%$\dn_2\in \matriz{n}{1}$\rm) \it is a full rank factorizations of $D$, 
then $A^{\parallel D}$ exists if and only if $\tr(AD) \neq 0$. Moreover, in
this case,
$$
A^{\parallel D}=\frac{1}{{\rm tr}(AD)}D.
$$
\end{cor}
\begin{proof} 
Let $D = \dn_1 \dn_2^*$ ($\dn_1\in \matriz{m}{1}$ and 
$\dn_2\in \matriz{n}{1}$) be a full rank factorization of $D$, 
According to Corollary \ref{c_rank_one} applied to the case $D=E$,
$\dn_2^* A \dn_1 = \tr(\dn_2^* A \dn_1) = \tr(A \dn_1 \dn_2^*) = \tr(AD)$.
\end{proof}

\indent Let $\matriz{n}{m}^{\parallel D}$ stand for the set of all matrices $A\in \ce_{n,m}$ such that $A$ is invertible along $D$, i.e.,
$$
\matriz{n}{m}^{\parallel D}=\{A\in \ce_{n,m}\colon A^{\parallel D}\hbox{ exists }\}.
$$
The following corollary proves that the set under consideration is nonempty.

\begin{cor} \label{cor4002}Let $D \in \ce_{m,n}$. Then,  $\matriz{n}{m}^{\parallel D}\neq\emptyset$. 
\end{cor}
\begin{proof} Apply Theorem \ref{thm4100}, Corollary \ref{cor305} and Corollary \ref{cor30200}.
\end{proof}

Observe that Corollary \ref{cor30200} gives an explicit representation of 
$\matriz{n}{m}^{\parallel D}$.

%%%%%%%%%%%%%%%%%%%%%%%%%%%%%%%%%%%

\section{Outer and inner inverses}

In the following theorem it will be proved that the notion introduced in Definition \ref{def3035} is an outer inverse.

\begin{thm}\label{outer} Let $A \in \ce_{n,m}$ and $D,E \in \ce_{m,n}$.
If $A$ is $(D, E)$-invertible, then $A^{\parallel (D, E)}$ is an outer inverse of $A$.
\end{thm}
\begin{proof} According to Definition \ref{def3035}, since $\rr(A^{\parallel (D, E)})\subseteq\rr(D)$,  there exists 
$M \in \ce_n$ such that $A^{\parallel (D, E)}=DM$. Thus,
$A^{\parallel (D, E)} A A^{\parallel (D, E)}=A^{\parallel (D, E)} A DM = DM =A^{\parallel (D, E)}$.
\end{proof}

\begin{cor}\label{corouter} Let $A \in \ce_{n,m}$ and $D\in \ce_{m,n}$. If $A$ is invertible along $D$, then $A^{\parallel D}$ is an outer inverse of $A$.
\end{cor}
\begin{proof} Apply Theorem \ref{outer} for the case $D=E$.
\end{proof}
 Given $A \in \matriz{n}{m}$ and  $D, E \in \matriz{m}{n}$ such that $\ADE$ exists, according to Corollary \ref{cor30200} and Theorem \ref{thm60000},
$$
\rk(\ADE) = \rk(D)=\rk(E) \leq \rk(A).
$$
In particular, when $A$, $D$ and $E$ are square matrices, $\ADE$ is nonsingular if and only if $D$ or $E$ are nonsingular.
In addition, if $\ADE$ is nonsingular, then $A$ is nonsingular. However,
if $A$ is nonsingular and $D,E$ are such that $\ADE$ exists, then it may be happen
that $D$, $E$, or $\ADE$ are singular. For example, take 
$$
A=I_2, \hskip.3truecm D=E=\mat{1}{0}{0}{0}.
$$
According to Theorem \ref{thm60000}, $\ADE$ exists
($\ADE=D$). However, it is possible to characterize when $\rk(A)=\rk(D)$. This characterization is linked with the following observation:
$\ADE$ is always an outer inverse of $A$  (Theorem \ref{outer}), but it is not necessarily an inner inverse of $A$. To prove this
characterization some preparation is needed first.

\begin{thm}\label{thm90000}Let $A\in \matriz{n}{m}$ and $D,E \in \matriz{m}{n}$ be such that $\ADE$ exists. Then, the following statements hold.
\begin{itemize}
\item[{\rm (i)}] $\rr(D) \oplus \kk(A) = \kk(A \ADE A-A)$.
\item[{\rm (ii)}] $\rk(A) = \rk(D) + \rk(A \ADE A-A)$.
\item[{\rm (iii)}] $\rr(A) + \kk(E) = \ce^n$ and $\rr(A) \cap \kk(E) = \rr(A \ADE A-A)$.
\end{itemize}
\end{thm}
\begin{proof}
The inclusion $\kk(A) \subseteq \kk(A \ADE A-A)$ is evident.
The equality $\ADE AD=D$ leads to $\rr(D) \subseteq \kk(A \ADE A-A)$.
If $\xn \in \rr(D) \cap \kk(A)$, then there exists
$\un \in \ce^n$ such that $\xn = D \un$, and therefore,
$\xn = D \un = \ADE AD \un = \ADE A \xn = \on$. Thus,
$\rr(D)\oplus \kk(A) \subseteq \kk(A \ADE A-A)$.

To prove the opposite inclusion, take $\yn \in \kk(A \ADE A-A)$. Now,
$A (\ADE A \yn - \yn) = \on$ and the
decomposition $\yn = (\yn - \ADE A(\yn)) + \ADE A(\yn)$
proves that $\kk(A \ADE A-A)\subseteq \rr(D) \oplus \kk(A)$
(because $\rr(\ADE)=\rr(D)$).

Statement (ii) follows from statement (i).

According to Remark \ref{rema40000} (i), $(A^*)^{\parallel (E^*, D^*)}$ exists, so that, according to statement (i) applied to $A^*$, $D^*$ and $E^*$,
$\kk(A^*) \cap \rr(E^*) = 0$. Then,
$$
[\rr(A) + \kk(E)]^\perp = \rr(A)^\perp \cap \kk(E)^\perp
= \kk(A^*) \cap \rr(E^*) = 0.
$$
Therefore, $\rr(A)+ \kk(E) = \ce^n$. The inclusion
$\rr(A \ADE A-A) \subseteq \rr(A) \cap \kk(E)$ follows from
$EA\ADE=E$. According to statement (ii), Theorem \ref{thm60000} and
\begin{equation*}
\begin{split} \dim (\rr(A) \cap \kk(E)) & = \dim \rr(A) + \dim \kk(E)- \dim(\rr(A)+\kk(E))\\
& = \rk(A)+n-\rk(D)-n = \rk(A \ADE A-A),
\end{split}
\end{equation*}
$\rr(A \ADE A-A)$ and $\rr(A) \cap \kk(E)$ have the same dimension.
Therefore, both subspaces are equal.
\end{proof}

The following corollary characterizes when $\ADE$
is an inner inverse of $A$.

\begin{cor}\label{cor100000}
Let $A\in \matriz{n}{m}$ and $D,E \in \matriz{m}{n}$ be such that $\ADE$ exists. Then,
the following statements are equivalent.
\begin{itemize}
\item[{\rm (i)}] $A\ADE A=A$.
\item[{\rm (ii)}] $\rk(A)=\rk(D)$.
\item[{\rm (iii)}] $\rr{(D)} \oplus \kk{(A)} = \ce^m$.
\item[{\rm (iv)}] $\rr{(A)} \oplus \kk{(E)} = \ce^n$.
\end{itemize}
\end{cor}
\begin{proof} Apply Theorem \ref{thm90000}.
\end{proof}

In the next corollary the case when the inverse along an element is an inner inverse will be presented.

\begin{cor}\label{cor200000}
Let $A\in \matriz{n}{m}$ and $D \in \matriz{m}{n}$ be such that $A^{\parallel D}$ exists. Then,
the following statements are equivalent.
\begin{itemize}
\item[{\rm (i)}] $AA^{\parallel D} A=A$.
\item[{\rm (ii)}] $\rk(A)=\rk(D)$.
\item[{\rm (iii)}] $\rr{(D)} \oplus \kk{(A)} = \ce^m$.
\item[{\rm (iv)}] $\rr{(A)} \oplus \kk{(D)} = \ce^n$.
\end{itemize}
\end{cor}
\begin{proof} Apply Corollary \ref{cor100000} to the case $D=E$.
\end{proof}

Let $A \in \matriz{n}{m}$ and $D,E \in \matriz{m}{n}$. Since $\ADE$ is an outer inverse, if it exists, 
$\ADE A$ and $A \ADE$ are idempotents. Now some properties of these idempotents will be studied.

\begin{thm}\label{propidem}
Let $A \in \matriz{n}{m}$ and $D,E \in \matriz{m}{n}$ be such that $\ADE$ exists.
\begin{enumerate}
\item[{\rm (i)}] $\ADE A$ and $A \ADE$ are idempotents, $\rr(\ADE A) = \rr(D)$, $\rr(A \ADE) = \rr(AD)$ and 
$\rk(D) = \rk(A \ADE) = \rk(\ADE A)=\rk(E)$.
\item[{\rm (ii)}] $\kk(A \ADE) = \kk(E)$ and $\kk(\ADE A) = \kk(EA)$.
\item[{\rm (iii)}] $\kk(\ADE A) = \kk(A)$
if and only if $\ADE$ is an inner inverse of $A$.
\item[{\rm (iv)}] 
$\rr(A \ADE) = \rr(A)$ if and only if $\ADE$ is an inner inverse of $A$.
\item[{\rm (v)}] $A \ADE$ is an orthogonal projector if and only if
$\rr(AD) = \rr(E^*)$.
\item[{\rm (vi)}] $\ADE A$ is an orthogonal projector if and only if
$\rr((EA)^*) = \rr(D)$.
\end{enumerate}
\end{thm}
\begin{proof} Since $\ADE$ is an outer inverse (Theorem \ref{outer}), $\ADE A$ and $A\ADE$are idempotents and $\rr(\ADE A)=\rr(\ADE)$.
Moreover, since $\rr(\ADE)=\rr(D)$, according to Theorem \ref{thm60000},
$$
\rk(E)=\rk(D)=\rk(\ADE)=\rk(\ADE A).
$$
In addition, according to Remark \ref{rema40000} (i) and what has been proved, 
$\rk(E)=\rk(E^*)=\rk((\ADE)^* A^*)=\rk(A\ADE)$. Moreover, 
$$
\rr(AD)=\rr(A\ADE AD)\subseteq \rr(A\ADE).
$$
However, since according to Corollary \ref{cor_rank}, $\rk(AD)=\rk(D)=\rk(A\ADE)$, $\rr(A\ADE)=\rr(AD)$.

Since $\ADE$ is an outer inverse, $\kk(A \ADE)=\kk(\ADE)=\kk(E)$. Note that since $EA=EA\ADE A$,
$\kk(\ADE A)\subseteq \kk(EA)$. However, since according to Corollary \ref{cor_rank}, $\rk(EA)=\rk(D)=\rk(\ADE A)$,
$\dim \kk(\ADE A)=\dim \kk(EA)$. Therefore, $\kk(\ADE A)=\kk(EA)$.

 Naturally, $\kk(A) \subseteq \kk(\ADE A)$. Since $\dim \kk(\ADE A) = m- \rk(\ADE A) = m-\rk(D)$
and $\dim \kk(A) = m - \rk(A)$, $\kk(A) = \kk(\ADE A)$ if and only if $\rk(A) = \rk(D)$. Now apply Corollary \ref{cor100000}.

 Note that $\ADE$ is an inner inverse of $A$
if and only if $(\ADE)^*$ is an inner inverse of $A^*$. Now, according to  statement (iii), this is equivalent to 
$\kk((\ADE)^* A^*)=\kk(A^*)$, which in turn is equivalent to $\rr(A \ADE) = \rr(A)$.

Observe that $A\ADE$ is an orthogonal projector
if and only if 
$$
\rr(AD)=\rr(A\ADE) = \kk(A\ADE)^\perp = \kk(E)^\perp = \rr(E^*).
$$

The proof of statement (vi) follows Remark \ref{rema40000} (i) and statement (v).
\end{proof}

In the next corollary the case of the inverse along a fixed matrix will be studied.

\begin{cor}\label{coridem_bis}
Let $A \in \matriz{n}{m}$ and $D\in \matriz{m}{n}$ be such that $A^{\parallel D}$ exists.
\begin{enumerate}
\item[{\rm (i)}] $A^{\parallel D} A$ and $A A^{\parallel D}$ are idempotents, $\rr(A^{\parallel D} A) = \rr(D)$, $\rr(A A^{\parallel D}) = \rr(AD)$ and 
$\rk(D) = \rk(A A^{\parallel D}) = \rk(A^{\parallel D} A)$.
\item[{\rm (ii)}] $\kk(A A^{\parallel D}) = \kk(D)$ and $\kk(A^{\parallel D} A) = \kk(DA)$.
\item[{\rm (iii)}] $\kk(A^{\parallel D} A) = \kk(A)$
if and only if $A^{\parallel D}$ is an inner inverse of $A$.
\item[{\rm (iv)}] 
$\rr(A A^{\parallel D}) = \rr(A)$ if and only if $A^{\parallel D}$ is an inner inverse of $A$.
\item[{\rm (v)}] $A A^{\parallel D}$ is an orthogonal projector if and only if
$\rr(AD) = \rr(D^*)$.
\item[{\rm (vi)}] $A^{\parallel D} A$ is an orthogonal projector if and only if
$\rr((DA)^*) = \rr(D)$.
\end{enumerate}
\end{cor}
\begin{proof} Apply Theorem \ref{propidem} to the case $D=E$.
\end{proof}

Given $A\in\ce_n$ such that $A$ is group invetible, according to \cite{bt} and \cite{rdd}, $A \core{A}$ and  $\cored{A} A$
are orthogonal projectors. In the next corollaries similar properties for several generalized inverses will be characterized using Corollary \ref{coridem_bis}.
Note that the following identities hold: $\rr(XY) = \rr(X)$ and
$\rr(XX^*) = \rr(X)$, where $Y$ is a nonsingular matrix and $X$ is any matrix. 
In addition, recall that a matrix $A\in\ce_n$ is said to be \it $EP$, \rm if  $AA^\dag = A^\dag A$.
It is well known that this condition is equivalent to $\rr(A)=\rr(A^*)$.

\begin{cor}\label{cor810000} Consider $A\in\ce_n$ a  group invertible matrix. The following statements holds.\par
\begin{itemize}
\item[{\rm (i)}] $\core{A} A$ is an orthogonal projector.
\item[{\rm (ii)}] $A \cored{A}$ is an orthogonal projector.
\item[{\rm (iii)}] $A$ is EP.
\end{itemize}
\end{cor}
\begin{proof} According to Theorem \ref{thmA} (i) and Theorem \ref{coridem_bis} (vi), 
$\core{A} A$ is an orthogonal projector if and only if $\rr((AA^* A)^*)=\rr(AA^*)$.  However, 
$\rr(AA^*)=\rr(A)$ and $\rr((AA^* A)^*)=\rr(A^*AA^*)=\rr(A^*A)=\rr(A^*)$.

Similarly, according to Theorem \ref{thmA} (ii) and Theorem \ref{coridem_bis} (v),
$A \cored{A}$ is an orthogonal projector  if and only $\rr(AA^*A)=\rr(A^*A)$.
However, $\rr(A^*A)=\rr(A^*)$ and $\rr(AA^*A)=\rr(AA^*)=\rr(A)$.
\end{proof}

\begin{cor}\label{cor700ter} Let $A \in \ce_n$ and consider  $M$, $N\in\ce_n$ nonsingular
 and positive. The following statements hold.
\begin{itemize}
\item[{\rm (i)}] $AA^\dag_{M,N}$ is an orthogonal projector
if and only if $\rr(A) =\rr(MA)$. In particular, if $M=I_n$, then $AA^\dag_{M,N}$ is an orthogonal projector.
\item[{\rm (ii)}] $A^\dag_{M,N} A$ is an orthogonal projector
if and only if $\rr(A^*) =\rr(N^{-1} A^*)$. In particular, if $N=I_n$, then $A^\dag_{M,N} A$ is an orthogonal projector.
\end{itemize}
\end{cor}
\begin{proof} According to Theorem \ref{thmA} (iii) and Theorem \ref{coridem_bis} (v), 
$A A^\dag_{M,N}$ is an orthogonal projector if and only if
$\rr(A N^{-1}A^* M) =\rr((N^{-1}A^*M)^*)$. This last condition
is equivalent to $\rr(A N^{-1}A^*) = \rr(M A)$. 
Since $N$ is nonsingular and positive, there exists a Hermitian and 
nonsingular matrix $Q$ such that $N^{-1} = Q^2$. Define $R=AQ$. Then, 
$AN^{-1}A^*= RR^*$, and thus, 
$$
\rr(A N^{-1}A^*) = \rr(RR^*) = \rr(R) =
\rr(AQ) = \rr(A).
$$

Similarly, according to Theorem \ref{thmA} (iii) and Theorem \ref{coridem_bis} (vi), $A^\dag_{M,N} A$ is an orthogonal projector if and only if 
$\rr((N^{-1}A^* M A)^*)=\rr(N^{-1}A^* M)$. This identity is equivalent to $\rr(A^*MA)=\rr(N^{-1}A^*)$. Since $M$ is nonsingular and positive, 
there exists a Hermitian and nonsingular matrix $P$ such that $M=P^2$. 
Then $A^*MA=A^*P^*PA=(PA)^*PA$, and thus,
$$
\rr(A^*MA)=\rr((PA)^*PA)=\rr((PA)^*)=\rr(A^*P^*)=\rr(A^*).
$$
\end{proof}

%%%%%%%%%%%%%%%%%%%%%%%%%%%%%%%%%%%%%%%%%%

\section{The outer inverse with prescribed range and null space}

In first place the relationship between the outer inverse with prescribed range and null space and the $(D, E)$-inverse
will be considered ($D$, $E\in\matriz{m}{n}$).

\begin{thm}\label{thm70000} 
Let $A \in \matriz{n}{m}$ and $D, E \in \matriz{m}{n}$. The following statements are equivalent.
\begin{enumerate}[{\rm (i)}]
\item The matrix $A$ is $(D, E)$-invertible.
\item The outer inverse $A^{(2)}_{\rr(D), \kk(E)}$ exists.
\end{enumerate}
Furthermore, in this case $\ADE=A^{(2)}_{\rr(D), \kk(E)}$.
\end{thm}
\begin{proof} Suppose that statement (i) holds and let $H=\ADE$. Then, according to 
Theorem~\ref{outer}, $H=HAH$. 
In addition, according to Definition \ref{def3035}, $\kk(H)=\kk(E)$ and $\rr(H)=\rr(D)$.
In particular, $A^{(2)}_{\rr(D), \kk(E)}$ exists and $A^{(2)}_{\rr(D), \kk(E)}=H$.

Now suppose that statement (ii) holds and let $L=A^{(2)}_{\rr(D), \kk(E)}$. In particular, 
$\rr(L) \subseteq \rr(D)$ and $\kk(E) \subseteq \kk(L)$. Since $L$ is an outer inverse of 
$A$ and  $\rr(L)=\rr(D)$, it is not difficult to prove that $LAD=D$. 
In addition, since $AL \xn - \xn \in \kk(L) = \kk(E)$, for all $\xn \in \ce^n$,
$EAL=E$. Therefore, $L=\ADE$.
\end{proof}

\begin{cor}\label{cor70011} 
Let $A \in \matriz{n}{m}$ and $D \in \matriz{m}{n}$. The following statements are equivalent.
\begin{enumerate}[{\rm (i)}]
\item The matrix $A$ is invertible along $D$.
\item The outer inverse $A^{(2)}_{\rr(D), \kk(D)}$ exists.
\end{enumerate}
Furthermore, in this case $A^{\parallel D}=A^{(2)}_{\rr(D), \kk(D)}$.
\end{cor}
\begin{proof} Apply Theorem \ref{thm70000}  to the case $D=E$.
\end{proof}

\indent Due to Theorem \ref{thm70000},  
the properties of the outer inverse with prescribed range and null space can be easily proved for the $(D, E)$-inverse
($D, E \in \matriz{m}{n}$).
Here only some of the most well known result are considered. Other results and the case of the inverse along a fixed matrix, i.e. when
$D=E$, are left to the reader.

\begin{cor}\label{cor70020}Let $A\in\mathbb{C}_{n, m}$ and $D$, $E\in\mathbb{C}_{m, n}$. $\ADE$
is the unique matrix $X$ that satifies the following equations:
$$
XAX=X, \hskip.3truecm AX=P_{\rr(AD), \kk(E)}, \hskip.3truecm XA=P_{\rr(D), \kk(EA)}.
$$
\end{cor}
\begin{proof}Apply Theorem \ref{thm70000} and \cite[Theorem 1]{WW}. Note that, if 
$\mathcal{T}=\rr(D)$ and $\mathcal{S}=\kk(E)$,
then $A(\mathcal{T})=\rr(AD)$ and $(A^*(\mathcal{S}^{\perp}))^{\perp}=\kk(EA)$.
\end{proof}

\begin{cor}\label{cor70030} Let $A\in\mathbb{C}_{n, m}$ and $D$, $E\in\mathbb{C}_{m, n}$. Suppose that there exists $G\in \mathbb{C}_{m, n}$
such that $\rr(G)=\rr(D)$ and $\kk(G)=\kk(E)$. If $A^{\parallel (D, E)}$ exists, then $AG\in \ce_m$ and $GA\in \ce_n$ are group invertible and
$$
A^{\parallel (D, E)}= G(AG)^\#=(GA)^\# G=[GA\mid_{\rr(G)}]^{-1}G.
$$
\end{cor}
\begin{proof}Apply Theorem \ref{thm70000},   \cite[Theorem 2.1]{W} and  \cite[Theorem 2.3]{W}.
\end{proof}

\begin{cor}\label{cor70040} Let $A\in\mathbb{C}_{n, m}$ and $D$, $E\in\mathbb{C}_{m, n}$. Suppose that there exists $G\in \mathbb{C}_{m, n}$
such that $\rr(G)=\rr(D)$ and $\kk(G)=\kk(E)$. If $A^{\parallel (D, E)}$ exists, then
$$
A^{\parallel (D, E)}=\lim_{\epsilon\to 0} (GA-\epsilon I_m)^{-1}G=\lim_{\epsilon\to 0} G(AG-\epsilon I_n)^{-1}.
$$
\end{cor}
\begin{proof} Apply Theorem \ref{thm70000} and   \cite[Theorem 2.4]{W}.
\end{proof}

\begin{cor}\label{cor70050} Let $A\in\mathbb{C}_{n, m}$ and $D$, $E\in\mathbb{C}_{m, n}$. Suppose that there exists $G\in \mathbb{C}_{m, n}$
such that $\rr(G)=\rr(D)$ and $\kk(G)=\kk(E)$. If $A^{\parallel (D, E)}$ exists, then
$$
A^{\parallel (D, E)}=\int_0^\infty {\rm exp}[-G(GAG)^*GAt]G(GAG)^*Gdt.
$$
\end{cor}
\begin{proof}
Apply Theorem \ref{thm70000} and   \cite[Theorem 2.2]{WD}.
\end{proof}

\begin{cor}\label{cor70060} Let $A\in\mathbb{C}_{n, m}$ and $D$, $E\in\mathbb{C}_{m, n}$. Suppose that there exists $G\in \mathbb{C}_{m, n}$
such that $\rr(G)=\rr(D)$ and $\kk(G)=\kk(E)$. If $A^{\parallel (D, E)}$ exists and the nonzero spectrum of $GA$ lies in the open left half plane, then
$$
A^{\parallel (D, E)}=-\int_0^\infty {\rm exp} (GAt)Gdt.
$$
\end{cor}
\begin{proof}
Apply Theorem \ref{thm70000} and   \cite{W2}.
\end{proof}

%%%%%%%%%%%%%%%%%%%%%%%%%%%%%%

\section{Continuity and differentiability}

\noindent First of all note that  if 
$A\in\mathbb{C}_{n, m}$ and $D$, $E\in \mathbb{C}_{m, n}$ are such that
$A^{\parallel (D, E)}$ exists and $D'$, $E'\in\mathbb{C}_{n, m}$ are such that $D=DD'D$ and $E=EE'E$, then according to Definition \ref{def3035}, 
$DD'A^{\parallel (D, E)}=A^{\parallel (D, E)}$ 
(because $\rr(A^{\parallel (D, E)})\subseteq \rr(D)$) and 
$A^{\parallel (D, E)}=A^{\parallel (D, E)}E'E$.
The last idendity can be easily derived from the fact that there exists a matrix $N\in \mathbb{C}_{m, m}$ such that $A^{\parallel (D, E)}=NE$
(because $\kk(E)\subseteq\kk(\ADE)$).\par

\indent In order to characterize the continuity of the $(D, E)$-inverse, a technical lemma is needed. \par

\begin{lem}\label{lem6000} Let $A$, $B\in\mathbb{C}_{n, m}$ and $D$, $E$, $F$, $G\in \mathbb{C}_{m, n}$ be such that 
$A^{\parallel (D, E)}$ and $B^{\parallel (F, G)}$ exist. Let $D'$, $E'$, $F'$ and $G'\in \mathbb{C}_{n, m}$ be such that $D=DD'D$, $E=EE'E$, $F=FF'F$ and $G=GG'G$.
Then
\begin{equation*}
\begin{split}
B^{\parallel (F, G)}-A^{\parallel (D, E)} & = 
B^{\parallel (F, G)}(G'G-E'E)(I_n-AA^{\parallel (D, E)}) + 
B^{\parallel (F, G)}(A-B)A^{\parallel (D, E)} \\
& \phantom{=} + (I_m-B^{\parallel (F, G)}B)(FF'-DD')A^{\parallel (D, E)}.\\
\end{split}
\end{equation*}
\end{lem}
\begin{proof}
Since $DD'A^{\| (D, E)}=A^{\| (D, E)}$ and $B^{\| (F, G)}BF=F$, then 
\begin{equation*}
\begin{split}
B^{\| (F, G)}BA^{\| (D, E)}-A^{\| (D, E)}&=-(I_m-B^{\| (F, G)}B)DD'A^{\| (D, E)} \\
&=[(I_m-B^{\| (F, G)}B)FF'-(I_m-B^{\| (F, G)}B)DD']A^{\| (D, E)}\\
&=(I_m-B^{\| (F, G)}B)(FF'-DD')A^{\| (D, E)}.\\
\end{split}
\end{equation*}

\indent In addition, since 
$B^{\| (F, G)}=B^{\| (F, G)}G'G$ and $EAA^{\| (D, E)}=E$,
\begin{equation*}
\begin{split}
B^{\| (F, G)}-B^{\| (F, G)}AA^{\| (D, E)} & = B^{\| (F, G)}G'G(I_n-AA^{\| (D, E)}) \\
&=B^{\| (F, G)}[G'G(I_n-AA^{\| (D, E)})-E'E(I_n-AA^{\| (D, E)})]\\
&=B^{\| (F, G)}(G'G-E'E)(I_n-AA^{\| (D, E)}).\\
\end{split}
\end{equation*}
Thus,
\begin{equation*}
\begin{split}
B^{\| (F, G)}-A^{\| (D, E)} & = 
B^{\| (F, G)}(G'G-E'E)(I_n-AA^{\| (D, E)}) +B^{\| (F, G)}AA^{\| (D, E)}\\
& \phantom{=} + (I_m-B^{\| (F, G)}B)(FF'-DD')A^{\| (D, E)}-B^{\| (F, G)}BA^{\| (D, E)}\\
& = B^{\| (F, G)}(G'G-E'E)(I_n-AA^{\| (D, E)})+ B^{\| (F, G)}(A-B)A^{\| (D, E)} \\
& \phantom{=} +(I_m-B^{\| (F, G)}B)(FF'-DD')A^{\| (D, E)}.\\
\end{split}
\end{equation*}
\end{proof}

\indent Next a result regarding the continuity of the $(D, E)$-inverse will be presented.\par

\begin{thm}\label{thm6001} 
Let $A\in\ce_{n,m}$ and $D, E\in \ce_{m,n}$ be such that $A^{\| (D,E)}$ exists and consider 
$(A_k)_{k\in\ene} \subset \ce_{n,m}$ and $(D_k)_{n\in\ene}$, $(E_k)_{n\in\ene} \subset \ce_{m,n}$ 
such that $A_k^{\| (D_n, E_n)}$ exists for each $k\in\ene$. Let $D', E' \in \ce_{n, m}$ and
$(D'_k)_{n\in\ene}$, $(E'_n)_{n\in\ene}\subset \ce_{n,m}$ be such that $D=DD'D$, $E=EE'E$, 
$D_k=D_kD'_k D_k$ and $E_k=E_kE'_kE_k$, for each $k\in\ene$. Suppose that $(A_k)_{k\in\ene}$,
 $(D_kD'_k)_{k\in\ene}$ and $(E'_kE_k)_{k\in\ene}$ converge to $A$, $DD'$ and 
$E'E$, respectively. Then, the following statememts are equivalent.\par
\begin{enumerate}[{\rm (i)}]
\item  $(A_k^{\| (D_k, E_k)})_{k\in\ene}$ converges to $A^{\parallel (D, E)}$.
\item The sequence $(A_k^{\| (D_k, E_k)})_{k\in\ene}\subset \ce_{n, m}$ is bounded.
\end{enumerate}
\end{thm} 
\begin{proof}Apply Lemma \ref{lem6000}.
\end{proof}

If the Moore-Penrose inverse is used, then a more general result can be presented.

\begin{thm}\label{thm6002} 
Let $A\in\ce_{n, m}$ and $D, E\in \ce_{m, n}$ be such that $A^{\| (D,E)}$ exists and consider 
$(A_k)_{k\in\ene}\subset \ce_{n,m}$ and $(D_k)_{k\in\ene}$, $(E_k)_{n\in\ene} \subset \ce_{m,n}$
such that $A_k^{\| (D_k, E_k)}$ exists for each $k\in\ene$. Suppose that $(A_k)_{k\in\ene}$ converges to $A$.  Then, the following statememts are equivalent.
\begin{enumerate}[{\rm (i)}]
\item $(A_k^{\| (D_k, E_k)})_{k\in\ene}$ converges to $A^{\parallel (D, E)}$.
\item The sequences $(D_k^\dag)_{k\in\ene}$ and $(E_k^\dag)_{k\in\ene}$ converge to $D^\dag$ and 
$E^\dag$, respectively, and the sequence $(A_k^{\| (D_k, E_k)})_{k\in\ene}\subset \ce_{n, m}$ is bounded. 
\end{enumerate}
\end{thm} 
\begin{proof} Suppose that $(A_k^{\| (D_k,E_k)})_{k\in\ene}$ converges to $A^{\| (D, E)}$. Then, 
$(A_k^{\| (D_k, E_k)}A_k)_{k\in\ene}$ converges to $A^{\| (D, E)}A$. Consequently, 
$$
\lim_{k\to\infty} \tr(A_k^{\| (D_k, E_k)}A_k)=\tr(A^{\| (D, E)}A).
$$
\noindent Since $A^{\| (D, E)}$ is an outer inverse (Theorem \ref{outer}), $A^{\| (D, E)}A$ 
is an idempotent. Thus, 
$$
\tr(A^{\| (D, E)}A)=\rk(A^{\| (D, E)}A)=\rk(A^{\| (D, E)}).
$$
Similarly, $\tr(A_k^{\| (D, E)}A)=\rk(A_k^{\| (D, E)})$, for $k\in\ene$. As a result, 
for sufficiently large $k\in\ene$, $\rk(A_k^{\| (D,E)})=\rk(A^{\| (D, E)})$. However, 
according to Theorem \ref{thm60000}
$$  
\rk(D_k)=\rk(E_k)=\rk(A_k^{\| (D, E)})=\rk(A^{\| (D, E)})=\rk(D)=\rk(E).
$$
Therefore, according to \cite{S}, $(D_k^\dag)_{k\in\ene}$ converges to $D^\dag$  and  $(E_k^\dag)_{k\in\ene}$ to $E^\dag$.
The remaining part of statement (ii) is evident.

If statement (ii) holds, then apply Theorem \ref{thm6001} with $D'=D^\dag$ and $E'=E^\dag$, $k\in\ene$.
\end{proof}

In the following corollary, the case of the inverse along a matrix will be considered.

\begin{thm}\label{thm6003} 
Let $A\in\ce_{n,m}$ and $D\in \ce_{m,n}$ be such that $A^{\| D}$ exists and consider 
$(A_k)_{k\in\ene} \subset \ce_{n,m}$ and $(D_k)_{k\in\ene}\subset \ce_{m, n}$ such that 
$A_k^{\| D_k}$ exists for each $k\in\ene$. Suppose that $(A_k)_{k\in\ene}$ converges to $A$.  
Then, the following statememts are equivalent.
\begin{enumerate}[{\rm (i)}]
\item $(A_k^{\| D_k})_{k\in\ene}$ converges to $A^{\parallel D}$.
\item The sequences $(D_k^\dag)_{k\in\ene}$ converges to $D^\dag$ and  the sequence 
$(A_k^{\| D_k})_{k\in\ene}\subset \ce_{n,m}$ is bounded. 
\end{enumerate}
\end{thm} 
\begin{proof} Apply Theorem \ref{thm6002} for the case $D=E$.
\end{proof}

Now the differentiability will be studied.

\begin{thm}\label{thm6004}
Let $J \subseteq \erre$ be an open set and consider $t_0 \in J$. Let functions
$ \mathcal{A}\colon J  \to \ce_{n,m}$ and  $\mathcal{D}$, 
$\mathcal{E}\colon J \to \ce_{m,n}$ be such that 
$\mathcal{A}(t)$ is $(\mathcal{D}(t), \mathcal{E}(t))$-invertible, for any $t \in J$, and $\mathcal{A}$, $\mathcal{D}$ and $\mathcal{E}$ are differentiable at $t_0$. Suppose that $f\colon J\to \ce_{m,n}$,
$f(t)= \mathcal{A}(t)^{\| (\mathcal{D}(t),\mathcal{E}(t))}$, is a bounded function in $J$ and that the functions
$\mathcal{D}$, $\mathcal{E}$ have local constant rank in $J$. Then, the function $f$ is differentiable at $t_0$ and
\begin{align*}
f'(t_0)&=\mathcal{A}(t_0)^{\| (\mathcal{D}(t_0), \mathcal{E}(t_0))}\left[\mathcal{G}'(t_0)\mathcal{E}(t_0)+\mathcal{G}(t_0)\mathcal{E}'(t_0)\right]
\left[I_n - \mathcal{A}(t_0) \mathcal{A}(t_0)^{\parallel (\mathcal{D}(t_0),\mathcal{E}(t_0))} \right] \\
&\phantom{=} + \left[I_n - \mathcal{A}(t_0)^{\| (\mathcal{D}(t_0),\mathcal{E}(t_0))}\mathcal{A}(t_0) \right]\left[\mathcal{D}'(t_0)\mathcal{F}(t_0)+\mathcal{D}(t_0)\mathcal{F}'(t_0)\right]
\mathcal{A}(t_0)^{\| (\mathcal{D}(t_0),\mathcal{E}(t_0))}\\
&\phantom{=} + \mathcal{A}(t_0)^{\| (\mathcal{D}(t_0),\mathcal{E}(t_0))}\mathcal{A}'(t_0) \mathcal{A}(t_0)^{\| (\mathcal{D}(t_0),\mathcal{E}(t_0))},
\end{align*}
\noindent where $\mathcal{F}$, $\mathcal{G}\colon J\to \ce_{n,m}$ are the functions $\mathcal{F}(t)=(\mathcal{D}(t))^\dag$ and 
$\mathcal{G}(t)=(\mathcal{E}(t))^\dag$.
\end{thm}
\begin{proof} Observe that according to Lemma \ref{lem6000}, for any $t \in J$, 
\begin{equation*}
\begin{split}
f(t)-f(t_0)  & =  \mathcal{A}(t)^{\parallel (\mathcal{D}(t), \mathcal{E}(t))} \left[\mathcal{E}(t)^\dag \mathcal{E}(t) - \mathcal{E}(t_0)^\dag \mathcal{E}(t_0)\right] 
\left[I_n - \mathcal{A}(t_0) \mathcal{A}(t_0)^{\parallel (\mathcal{D}(t_0),\mathcal{E}(t_0))} \right] \\ 
& \phantom{=} + 
\left[I_n - \mathcal{A}(t)^{\| (\mathcal{D}(t),\mathcal{E}(t))}\mathcal{A}(t) \right] \left[ \mathcal{D}(t)\mathcal{D}(t)^\dag - \mathcal{D}(t_0)\mathcal{D}(t_0)^\dag \right] 
\mathcal{A}(t_0)^{\| (\mathcal{D}(t_0),\mathcal{E}(t_0))} \\ 
& \phantom{=} 	+ \mathcal{A}(t)^{\parallel (\mathcal{D}(t),\mathcal{E}(t))} \left[\mathcal{A}(t_0)-\mathcal{A}(t) \right] \mathcal{A}(t_0)^{\| (\mathcal{D}(t_0),\mathcal{E}(t_0))}. \\
\end{split}
\end{equation*}
\indent Now, according to \cite{S}, the functions $\mathcal{F}$, $\mathcal{G}\colon J\to \ce_{n,m}$, $\mathcal{F}(t)=(\mathcal{D}(t))^\dag$ and 
$\mathcal{G}(t)=(\mathcal{E}(t))^\dag$ are continuous. Consequently, according to Theorem \ref{thm6003}, 
$$
\lim_{t \to t_0}\mathcal{A}(t)^{\parallel (\mathcal{D}(t),\mathcal{E}(t))}  \frac{\left[\mathcal{A}(t_0)-\mathcal{A}(t) \right]}{t-t_0} \mathcal{A}(t_0)^{\| (\mathcal{D}(t_0),\mathcal{E}(t_0))}=
 \mathcal{A}(t_0)^{\| (\mathcal{D}(t_0),\mathcal{E}(t_0))}\mathcal{A}'(t_0) \mathcal{A}(t_0)^{\| (\mathcal{D}(t_0),\mathcal{E}(t_0))}.
$$
\noindent In addition, according to \cite{GP}, the functions $\mathcal{F}$, $\mathcal{G}\colon J\to \ce_{n,m}$ are also differentiable. Thus

\begin{equation*}
\begin{split}
\lim_{t \to t_0}&\mathcal{A}(t)^{\parallel (\mathcal{D}(t), \mathcal{E}(t))} \frac{\left[\mathcal{E}(t)^\dag \mathcal{E}(t) - \mathcal{E}(t_0)^\dag \mathcal{E}(t_0)\right]}{t-t_0} 
\left[I_n - \mathcal{A}(t_0) \mathcal{A}(t_0)^{\parallel (\mathcal{D}(t_0),\mathcal{E}(t_0))} \right] =\\ 
&\mathcal{A}(t_0)^{\parallel (\mathcal{D}(t_0), \mathcal{E}(t_0))}\left[\mathcal{G}'(t_0)\mathcal{E}(t_0)+\mathcal{G}(t_0)\mathcal{E}'(t_0)\right]
\left[I_n - \mathcal{A}(t_0) \mathcal{A}(t_0)^{\parallel (\mathcal{D}(t_0),\mathcal{E}(t_0))} \right]. 
\end{split}
\end{equation*}
Similarly,
\begin{equation*}
\begin{split}
\lim_{t \to t_0}&\left[I_n - \mathcal{A}(t)^{\| (\mathcal{D}(t),\mathcal{E}(t))}\mathcal{A}(t) \right]  \frac{\left[ \mathcal{D}(t)\mathcal{D}(t)^\dag - \mathcal{D}(t_0)\mathcal{D}(t_0)^\dag \right]}{t-t_0}
\mathcal{A}(t_0)^{\| (\mathcal{D}(t_0),\mathcal{E}(t_0))}= \\ 
&\left[I_n - \mathcal{A}(t_0)^{\| (\mathcal{D}(t_0),\mathcal{E}(t_0))}\mathcal{A}(t_0) \right]\left[\mathcal{D}'(t_0)\mathcal{F}(t_0)+\mathcal{D}(t_0)\mathcal{F}'(t_0)\right]
\mathcal{A}(t_0)^{\| (\mathcal{D}(t_0),\mathcal{E}(t_0))}.
\end{split}
\end{equation*}
\end{proof}

Now the differentiability of the inverse along a matrix will be studied.

\begin{cor}\label{cor6005}
Let $J \subseteq \erre$ be an open set and consider $t_0 \in J$. Let functions
$\mathcal{A}\colon J  \to \ce_{n,m}$ and  $\mathcal{D}\colon J \to \ce_{m,n}$ be such that 
$\mathcal{A}(t)$ is invertible along $\mathcal{D}(t)$ for any $t \in J$, and $\mathcal{A}$ and  $\mathcal{D}$ are differentiable at $t_0$. Suppose that $f\colon J\to \ce_{m,n}$,
$f(t)= \mathcal{A}(t)^{\| \mathcal{D}(t)}$, is a bounded function in $J$ and that the function
$\mathcal{D}$ has local constant rank in $J$. Then, the function $f$ is differentiable at $t_0$ and
\begin{align*}
f'(t_0)&=\mathcal{A}(t_0)^{\parallel \mathcal{D}(t_0)}\left[\mathcal{F}'(t_0)\mathcal{D}(t_0)+\mathcal{F}(t_0)\mathcal{D}'(t_0)\right]
\left[I_n - \mathcal{A}(t_0) \mathcal{A}(t_0)^{\parallel \mathcal{D}(t_0)} \right] \\
& \phantom{=} + \left[I_n - \mathcal{A}(t_0)^{\| \mathcal{D}(t_0)}\mathcal{A}(t_0) \right]\left[\mathcal{D}'(t_0)\mathcal{F}(t_0)+\mathcal{D}(t_0)\mathcal{F}'(t_0)\right]
\mathcal{A}(t_0)^{\| (\mathcal{D}(t_0)}\\
&\phantom{=} + \mathcal{A}(t_0)^{\| \mathcal{D}(t_0)}\mathcal{A}'(t_0) \mathcal{A}(t_0)^{\| \mathcal{D}(t_0)},
\end{align*}
\noindent where $\mathcal{F}\colon J\to \ce_{n,m}$ is the function $\mathcal{F}(t)=(\mathcal{D}(t))^\dag$.
\end{cor}
\begin{proof}Apply Theorem \ref{thm6004} for the case $D=E$.
\end{proof}

%%%%%%%%%%%%%%%%%%%%%%%%%%%%%%%%%%%%%%%%
\section{Explicit computations} \label{s4}

In this section some explicit ways to compute $\ADE$ will be given.

\begin{thm}\label{th24}
Let $A \in \matriz{n}{m}$, $D,E \in \matriz{m}{n}$, $r=\rk(D)$ and $s=\rk(E)$.
If $\{ {\bf v}_1, \ldots, {\bf v}_r \}$ is a basis of $\rr(D)$ and
$\{ {\bf w}_1, \ldots, {\bf w}_{n-s} \}$ is a basis of $\kk(E)$, then the following affirmations
are equivalent:
\begin{itemize}
\item[{\rm (i)}] $\ADE$ exists.
\item[{\rm (ii)}] The matrix $[A {\bf v}_1 \ \cdots \ A {\bf v}_r \ {\bf w}_1 \ \cdots
\ {\bf w}_{n-s}]$ is nonsingular.
\end{itemize}
In this situation,
$$
\ADE = \left[ {\bf v}_1 \ \cdots \ {\bf v}_r \ {\bf 0} \ \cdots \ {\bf 0} \right]
\left[A {\bf v}_1 \ \cdots \ A {\bf v}_r \ {\bf w}_1 \ \cdots
\ {\bf w}_{n-s}\right]^{-1}.
$$
\end{thm}
\begin{proof} If statement (i) holds, then according to Theorem \ref{thm60000}, $\rk(D)=\rk(E)$. 
Let $X_1 = [A \vn_1 \ \cdots \ A \vn_r]$ and $X_2 = [\wn_1 \ \cdots \ \wn_{n-r}]$.
Observe that $n-r=\rk(X_2)$ because $\{ \wn_i \}_{i=1}^{n-r}$ is a basis. According to 
Theorem~\ref{thm3036} and Theorem \ref{thm3033}, $\rk(X)=\rk(X_1)+\rk(X_2)$ (because
$\rr(AD) \cap \kk(E)=0$). Since $\{ A \vn_i \}_{i=1}^r$ span $\rr(AD)$ and
$r=\rk(D)=\rk(AD) = \dim \rr(AD)$, the vectors $\{ A \vn_i \}_{i=1}^r$
are linearly independent, and thus, $r=\rk(X_1)$. Therefore, $n=\rk(X)$ and by recalling
that $X \in \ce_n$, the nonsingularity of $X$ is obtained.

Suppose that statement (ii) holds. Since the matrix in statement (ii) must be square, $\rk(D)=r=s=\rk(E)$. 
In addition, since the matrix in statement (ii) is invertible, $\rk(AD)=\rk(D)$ and $\rr(AD)\oplus \kk(E)=\ce^n$.
Consequently, according to Theorem \ref{thm3033} and Theorem \ref{thm3036}, $\ADE$ exists.

Now let ${\bf v}$ be any arbitrary vector in $\rr(D)$  and let ${\bf x} \in \ce^n$ be such that
${\bf v} = D\bf x$. According to Definition \ref{def3035}, 
$\ADE A {\bf v} = \ADE A D {\bf x} =
D {\bf x} = {\bf v}$. In addition, $\ADE {\bf w} = {\bf 0}$ for any ${\bf w} \in \kk(E)$. Therefore,
$$
\ADE
\left[A {\bf v}_1 \ \cdots \ A {\bf v}_r \ {\bf w}_1 \ \cdots
\ {\bf w}_{n-r}\right] =
\left[ {\bf v}_1 \ \cdots \ {\bf v}_r \ {\bf 0} \ \cdots \ {\bf 0} \right].
$$
\end{proof}

Next {\tt m-file} that can be executed in Matlab or in Octave shows how Theorem~\ref{th24}
can be used to compute $\ADE$.

\begin{verbatim}
function J = pseudo(A,D,E)
[n m] = size(A);
r = rank(D);
s = rank(E);
E1 = null(E); % An orthonormal basis of N(E)
D1 = orth(D); % An orthonormal basis of R(D)
aux = [A*D1 E1];
if not(r==s)
   disp('There does not exist the pseudoinverse')
   disp('because rank(D) is not equal to rank(E)')
else
  if det(aux)==0
   disp('There does not exist the pseudoinverse')
   disp('because the matrix of Th. 4.1 is singular')
   else
     J=[D1 zeros(n,n-r)]*inv(aux);
   end
end
\end{verbatim}

\begin{thm}\label{t007}
Let $A \in \matriz{n}{m}$ and $D, E \in \matriz{m}{n}$ be such that 
$\ADE$ exists. Let $r=\rk(D)=\rk(E)=\rk(EAD)$. Let $P\in \ce_m$
and $Q \in \ce_n$ be two nonsingular matrices such that
$$
PEADQ=\mat{I_r}{0}{0}{0}.
$$
Then
\begin{equation} \label{ASD}
PE=\begin{bmatrix} X \\ 0 \end{bmatrix} \qquad \text{and} \qquad
DQ=\begin{bmatrix} Y &  0 \\ \end{bmatrix},
\end{equation}
where $X \in \matriz{r}{n}$,
$Y \in \matriz{m}{r}$. Furthermore, $\ADE=YX$.
\end{thm}
\begin{proof}
Write $P$ and $Q$ as
\begin{equation*} %\label{ASD2}
P=\begin{bmatrix} P_1 \\ P_2 \end{bmatrix} \quad \text{and} \quad
Q=\begin{bmatrix} Q_1 &  Q_2 \end{bmatrix},
\end{equation*}
where $P_1 \in \matriz{r}{m}$,
$P_2 \in \matriz{m-r}{m}$, $Q_1 \in \matriz{n}{r}$ and $Q_2 \in \matriz{n}{n-r}$.
Now 
\begin{equation}\label{peadq}
\mat{I_r}{0}{0}{0} = PEADQ =
\begin{bmatrix} P_1 \\ P_2 \end{bmatrix} EAD \begin{bmatrix} Q_1 & Q_2 \end{bmatrix} =
\mat{P_1EADQ_1}{P_1EADQ_2}{P_2EADQ_1}{P_2EADQ_2},
\end{equation}
which implies $P_1EADQ_2=0$, $P_2EADQ_1=0$ and $P_2EADQ_2=0$. Therefore,
$$
P_2EADQ = P_2EAD[Q_1 \ Q_2] = [P_2EADQ_1 \ P_2EADQ_2]=0.
$$
The nonsingularity of $Q$ leads to $P_2EAD=0$. In a similar way,
$EADQ_2=0$.

Since $\rk(D)=\rk(E)=\rk(EAD)$, the equalities 
$\mathcal{R}(EAD)=\mathcal{R}(E)$ and $\mathcal{N}(EAD)=\mathcal{N}(D)$ are obtained.
In addition, since $EADQ_2=0$ and $\mathcal{N}(EAD)=\mathcal{N}(D)$, it can be deduced $DQ_2=0$.
Since $(EAD)^*P_2^* = (P_2EAD)^* = 0$, any column of $P_2^*$ 
belongs to $\kk((EAD)^*) = \rr(EAD)^\perp = \rr(E)^\perp = \kk(E^*)$,
and therefore, $E^*P_2^*=0$, i.e., $P_2E=0$.
If $Y=DQ_1$ and $X=P_1E$, then (\ref{ASD}) holds.

Now it will be proved that $YX$ satisfies Definition \ref{def3035}.
First, observe that $XAY = P_1E A DQ_1 = I_r$.
Now,
by \eqref{ASD}
$$
YXADQ = YXA[ Y \ 0] = [YXAY \ 0] = [Y \ 0] = DQ,
$$
and the nonsingularity of $Q$ leads to $YXAD=D$. Similarly,
$$
PEAYX = \begin{bmatrix} X \\ 0 \end{bmatrix} AYX = \begin{bmatrix} XAYX \\ 0 \end{bmatrix} =
\begin{bmatrix} X \\ 0 \end{bmatrix} = PE,
$$
and thus, $EAYX=E$. Since $YX = DQ_1X$, te inclusion $\rr(YX) \subseteq \rr(D)$ can be obtained.
In addition, since $YX = YP_1E$, it can be deduced $\kk(E) \subseteq \kk(YX)$.
\end{proof}

\begin{rema}\label{rem100000}\rm
Observe that it is possible to use either the Gaussian elimination method or the singular value decomposition of $EAD$ to determine $P$ and $Q$. Let $r=\rk(EAD)$.
\begin{itemize}
\item[{\rm (i)}] By using the Gauss-Jordan elimination,
there exist an elementary row operation matrix $P \in \matriz{m}{m}$ and an
elementary column operation matrix $Q \in \matriz{n}{n}$,
such that $PEADQ=\mat{I_r}{0}{0}{0}$.
\item[{\rm (ii)}] Let $EAD=USV^*$ be the singular value decomposition of $EAD$, where $S=\Sigma\oplus0$,
$\Sigma={\rm diag}(\sigma_1, \ldots, \sigma_r)$. Hence, $U^*EADV=\Sigma\oplus0$, which implies
$$
(\Sigma^{-1/2}\oplus I_{m-r})U^*EADV(\Sigma^{-1/2}\oplus I_{n-r})=\mat{I_r}{0}{0}{0}.
$$
Let $P=(\Sigma^{-1/2}\oplus I_{m-r})U^*$ and $Q=V(\Sigma^{-1/2}\oplus I_{n-r})$. It is 
easy to see that $P$ and $Q$ are nonsingular.
\end{itemize}
\end{rema}

Theorem \ref{t007} and Remark \ref{rem100000} (i) yield an elimination method  to compute $\ADE$, which is presented as follows.

\begin{algorithm}[H]
\caption{\textsc{Compute the $(D, E)$-inverse}.}\label{algo:sylvester}
\KwIn{$A \in \matriz{n}{m}$, $D, E \in \matriz{m}{n}$ with $\rk(D)=\rk(E)=\rk(EAD)=r$.}
\KwOut{$\ADE$.}
\begin{enumerate}
\item Execute elementary row operations on the first $m$ rows of the block matrix
\begin{equation*} %\label{z}
G = \begin{bmatrix} EAD & E \\ D & 0 \end{bmatrix}
\end{equation*}
to get
\begin{equation*} %\label{z1}
G_1 = \begin{bmatrix}
\begin{bmatrix}
W \\ 0 \end{bmatrix} & \begin{bmatrix} X \\ 0 \end{bmatrix} \\ D & 0
\end{bmatrix}.
\end{equation*}
\item  Execute elementary column operations on the first $m$ columns of the  block matrix $G_1$ to get
\begin{equation*} %\label{z2}
G_2 = \begin{bmatrix}
\begin{bmatrix} I_r & 0 \\ 0 & 0 \end{bmatrix} & \begin{bmatrix} X \\ 0 \end{bmatrix} \\
\begin{bmatrix} Y & 0 \end{bmatrix} & 0
\end{bmatrix}.
\end{equation*}
\item $\ADE=YX$.
\end{enumerate}
\end{algorithm}

Theorem \ref{t007} and Remark \ref{rem100000} (ii) yield a more stable numerical method based
on the SVD to compute $\ADE$, which is shown as follows.

\begin{algorithm}[H]
\caption{\textsc{Compute the $(D, E)$-inverse}.}\label{algo:sylvester}
\KwIn{$A \in \matriz{n}{m}$, $D, E \in \matriz{m}{n}$ with $\rk(D)=\rk(E)=\rk(EAD)=r$.}
\KwOut{$\ADE$.}
\begin{enumerate}
\item Compute the SVD of $EAD$, i.e., $EAD=USV^*$.

\item $T=S(1:r,1:r)$, $M=T^{-1/2} \oplus I_{m-r}$, $N=T^{-1/2} \oplus I_{n-r}$.

\item $P=MU^*$, $Q=VN$.

\item $X=PE$, $Y=DQ$.

\item $\ADE=Y(1:m,1:r)\cdot X(1:r,1:n)$.
\end{enumerate}
\end{algorithm}

Next {\tt m-file} shows how Theorem \ref{t007} and the SVD can be used to compute $|ADE$.

\begin{verbatim}
function J = pseudo(A,D,E)
[n m] = size(A);
r = rank(D);
s = rank(E);
t = rank(E*A*D); 
if not(r==s)
   disp('There does not exist the pseudoinverse')
   disp('because rank(D) is not equal to rank(E)')
else
  if not(s==t)
   disp('There does not exist the pseudoinverse')
   disp('because rank(D)=rank(E) but not equal to rank(EAD)')
   else
     [U S V] = svd(E*A*D);
     T = S(1:r,1:r)
     M = [T^(-1/2) zeros(r,m-r); zeros(m-r,r) eye(m-r)]; 
     N = [T^(-1/2) zeros(r,n-r); zeros(n-r,r) eye(n-r)];
     P = M*U';
     Q = V*N;
     X = P*E; 
     Y = D*Q;
     J = Y(1:m,1:r)*X(1:r,1:n);   
   end
end
\end{verbatim}

%\section*{Acknowledgements} 

\vskip.3truecm
\noindent Julio Ben\'{\i}tez\par
\noindent E-mail address: jbenitez@mat.upv.es
\vskip.3truecm
\noindent Enrico Boasso\par
\noindent E-mail address: enrico\_odisseo@yahoo.it 
\vskip.3truecm
\noindent Hongwei Jin\par
\noindent E-mail address: hw-jin@hotmail.com 
\end{document}